\documentclass[12pt]{amsart}
\usepackage[utf8]{inputenc}
\usepackage{amsmath,amsthm,amsfonts,amssymb,mathabx,epsfig}
\usepackage{mathtools}
\usepackage[all]{xy}
\usepackage{tikz-cd}
\usepackage[shortlabels]{enumitem}
\usepackage{verbatim}
\usepackage{multirow}
\usepackage{cite}
\usepackage{setspace}
\usepackage{musicography}
\usepackage{imakeidx}
\usepackage{blindtext}
\usepackage{caption}
\usepackage{lscape}
\usepackage[a4paper,width=172mm,top=30mm,bottom=25mm]{geometry}

\newcommand{\R}{{\mathbb{R}}}
\newcommand{\C}{{\mathbb{C}}}
\newcommand{\K}{{\mathbb{K}}}

\newcommand{\J}{\mathcal{J}}
\newcommand{\F}{\mathcal{F}}

\newcommand{\X}{\mathfrak{X}}
\newcommand{\ra}{\rangle}
\newcommand{\la}{\langle}

\newcommand{\CM}{C^{\infty}(M)}

\newcommand{\lie}{\mathcal{L}}

\newcommand{\T}{\mathbb{T}}

\newcommand{\cL}{\overline{L}}
\newcommand{\CMC}{C^{\infty}(M,\mathbb{C})}
\newcommand{\CMR}{C^{\infty}(M,\mathbb{R})}
\newcommand{\TCM}{T_{\mathbb{C}}M}
\newcommand{\TCN}{T_{\mathbb{C}}N}
\newcommand{\CCM}{T^*_{\mathbb{C}}M}

\newcommand{\TTCM}{\mathbb{T}_{\mathbb{C}}M}

\newcommand{\st}{\:\:|\:\:}

\newcommand{\px}{\frac{\partial}{\partial x}}
\newcommand{\pxi}{\frac{\partial}{\partial x_i}}
\newcommand{\pxj}{\frac{\partial}{\partial x_j}}

\newcommand{\py}{\frac{\partial}{\partial y}}

\newcommand{\pyk}{\frac{\partial}{\partial y_k}}
\newcommand{\pz}{\frac{\partial}{\partial z}}

\newcommand{\bemol}{\musFlat{}}
\newcommand{\shrp}{\musSharp{}}

\newcommand{\deltpi}{\Delta_{\pi}}
\newcommand{\Dpi}{D_{\pi}}

\newtheorem{theorem}{Theorem}[section]
\newtheorem{proposition}[theorem]{Proposition}
\newtheorem{corollary}[theorem]{Corollary}
\newtheorem{lemma}[theorem]{Lemma}

\newtheorem{remark}[theorem]{Remark}
\newtheorem{examples}[theorem]{Examples}
\newtheorem{example}[theorem]{Example}

\theoremstyle{definition}
\newtheorem{definition}[theorem]{Definition}
\newtheorem{innercustomthm}{Theorem}
\newenvironment{customthm}[1]
{\renewcommand\theinnercustomthm{#1}\innercustomthm}
{\endinnercustomthm}

\DeclareMathOperator{\rea}{re}
\DeclareMathOperator{\ima}{im}

\DeclareMathOperator{\order}{order}

\DeclareMathOperator{\Ann}{Ann}

\DeclareMathOperator{\re}{Re}
\DeclareMathOperator{\rk}{rank}

\DeclareMathOperator{\pr}{pr}

\DeclareMathOperator{\gr}{graph}
\title{On the geometry of complex Poisson bivectors}
\author[D. Aguero]{Dan Aguero}
\address{Scuola Internazionale Superiore di Studi Avanzati (SISSA), Via Bonomea 265, 34136 Trieste, Italia}
\email{dagueroc@sissa.it}
\begin{document}
\begin{abstract}
    We study the geometry of complex Poisson bivectors over smooth manifolds. We show that under mild regularity conditions any complex Poisson bivector has associated a complex presymplectic foliation. After that, we use techniques of Dirac geometry to provide a more concise description of this complex presymplectic foliation. Moreover, we introduce two new classes of structures: quasi-real Poisson and quasi-real Dirac structures. In the last part, we focus on the normal form of complex Poisson bivectors. Under certain regularity, we provide a normal form theorem for complex Poisson structures along certain kinds of submanifolds. 
\end{abstract}
\maketitle
\tableofcontents
\addtocontents{toc}{\protect\setcounter{tocdepth}{1}}
\section{Introduction}
Poisson structures were first considered by Siméon Denis Poisson in the nineteenth century by the introduction of the Poisson bracket in mechanics. Later, Sophus Lie employed Poisson geometry in his research about transformation groups, which gave rise to today's Lie theory. Almost a century later André Lichnerowicz in his seminal paper ``Les varietés de Poisson et leurs algèbres de Lie associées'' (\cite{lichnerowicz1977varietes}) provided a more modern treatment of Poisson structures by introducing Poisson bivectors and the celebrated condition $[\pi,\pi]=0$.

Today, Poisson structures play an important role in the interaction between differential geometry and physics. The most commonly studied type of Poisson structure is the class of smooth Poisson structures on smooth manifolds. Nevertheless,
Poisson brackets defined over spaces of functions different from smooth functions have also been considered, for example, on the space of holomorphic functions of a holomorphic manifold (holomorphic Poisson structures), see \cite{laurent2008holomorphic}; on the structure sheaf of a scheme (Poisson schemes), see \cite{polishchuk1997algebraic}; and on the structure sheaf of a supermanifold or graded manifold \cite{cantrijn1991introduction, del2015geometric}.

A {\em complex Poisson structure} on a smooth manifold $M$ is a complex Lie algebra bracket defined on $\CMC$
\begin{equation*}
    \{\cdot,\cdot\}:\CMC\times\CMC\to \CMC
\end{equation*}
that satisfies the Leibniz identity. Equivalently, a complex Poisson structure $ \{\cdot,\cdot\}$ can be viewed as a complex bivector $\pi=\pi_1+i\pi_2\in \Gamma(\wedge^2\TCM)$ satisfying the Lichnerowicz identity $[\pi,\pi]=0$, where $[\cdot,\cdot]$ denotes the complexified Schouten bracket. The shift from real to complex coefficients, make these sort of structures more difficult and tricky to study. 

The class of complex Poisson structures considerably enlarges the class of Poisson structures. It includes, of course, all Poisson structures as well as holomorphic Poisson structures and the already broad class of bi-Hamiltonian structures. On the other hand, complex Poisson structures form a special subclass of complex Dirac structures (Dirac structures on $(TM\oplus T^*M)_{\C}$). Hence, understanding complex Poisson structures would provide us with a better insight about the class of complex Dirac structures.

For instance, we shall see that the geometry associated to complex Poisson brackets is expressed in terms of complex distributions and vector bundles equipped with a complex symplectic structure. By applying techniques from complex Dirac geometry and complex Lie algebroids, we shall extract information that is expressed in terms of real structures, i.e., real distributions and real geometrical structures.

The goal of this article is to provide a description of complex Poisson structures, mainly from the perspective of Dirac geometry. For that reason we do a extensive use of complex Dirac structures, see \cite{aguero2022complex, aguero}. The first natural invariant associated to a Poisson structure is its underlying symplectic foliation. This leads to our first question:

\medskip
\textbf{Question 1:} {\em Is there an analog of the underlying symplectic foliation in the context of complex Poisson structures?}

\medskip
At the infinitesimal level, a complex Poisson structure gives rise to a complex symplectic distribution $(E_{\pi}, \omega_{\rea}+i\omega_{\ima})$ (see Proposition \ref{cxsymplecticfamily}), where $E_{\pi}=\ima\pi^{\shrp}$. However, the issues appear at the global level, since it is not clear to what object complex distributions integrate. We prove that when the distribution of real elements $E_{\pi}\cap TM$ is integrable, the leaves inherits a complex presymplectic two-form (see Proposition \ref{cx_presymp_leaf}). This problem can also be approached from the perspective of Dirac geometry, since the graph of a complex Poisson structure is a complex Dirac structure. It is known that we can correlate a real lagrangian family $\widehat{L}$ to a complex Dirac structure $L$.  When this lagrangian family is smooth, its presymplectic foliation agrees, up to a sign, with the presymplectic foliation $(E_{\pi}\cap TM, \omega_{\rea})$ (Proposition \ref{same_fol}).   

Motivated by the discrepancy between the complex symplectic distribution and the tangent of the complex presymplectic foliation, we introduce a new class of complex Poisson bivector for which these two distributions coincide, which we refer to as {\em quasi-real Poisson structures}. Under a slight modification of the proof of the local form theorem for Lie algebroids in \cite{Dufour2001NormalFF}, we are able to prove that the image of a quasi-real Poisson bivector is always integrable. We go further and introduce the notion of {\em quasi-real lagrangian families} with the same logic of quasi-real Poisson structures. We correlate a quasi-real lagrangian family to any complex Dirac structure. To do so, we first identify a second Dirac structure linked to a complex Dirac structure $L$, which we denote by $\widecheck{L}$. Second, we introduce the operation of {\em complex tangent sum} of two real lagrangian families that produces a complex lagrangian family from two real lagrangian families. 

Finally, we associate a quasi-real lagrangian family $\widetilde{L}$ to any complex Dirac structure $L$ that is given by the complex sum of $\widehat{L}$ with $\widecheck{L}$.

Using the quasi-real lagrangian family associated to a complex Poisson structure, we provide a Dirac-geometric characterization of the complex presymplectic foliation:
\begin{customthm}{7.18}
Let $\pi$ be a complex Poisson structure. Then,
\begin{equation*}
    \widetilde{L_{\pi}}=L((\deltpi)_\C,\Omega|_{(\deltpi)_{\C}}),
\end{equation*}
where $(\deltpi, \Omega|_{(\deltpi)_{\C}})$ is the complex presymplectic foliation associated to $\pi$. That is, all the information of the complex presymplectic foliation is encoded in $\widetilde{L_{\pi}}$. Moreover, if $\pi$ is strongly regular, i.e., $\ima\pi^{\shrp}$ and $\ima\pi^{\shrp}\cap TM$ are regular distributions, then $\widetilde{L_{\pi}}$ is a complex Dirac structure.
\end{customthm}

\medskip

In the context of complex Dirac structures, one of the most important invariants is the {\em real index}, i.e., the dimension of the space of purely real elements inside a complex distribution. This leads to our second question:

\medskip
\textbf{Question 2:} {\em How does the real index work with respect to complex Poisson structures?}

\medskip
 When approaching complex Poisson bivectors from the perspective of Dirac geometry, it is natural to ask about the real index and the conditions under which a complex Poisson structure defines a generalized complex structure. We prove that the real index of the graph of a complex Poisson bivector $\pi_1+i\pi_2$ is $\rk\ker\pi_2$ (see Lemma \ref{realpart}).

We also study a reduction scheme for complex Poisson bivector. Our main goal is to provide conditions under which the real index can be annihilated, obtaining a generalized complex structure on a quotient space (see Proposition \ref{reductionbykernel2}).

In the context of zero real index, we prove the following:
\begin{customthm}{8.10}
Let $\pi=\pi_1+i\pi_2$ be a complex Poisson bivector. Then $L_{\pi}$ is a generalized complex structure if and only if $\pi_2$ is invertible. In this case, the generalized complex map $\J_{\pi}: \T M\to \T M$ associated to $L_{\pi}$ is given by:
\begin{equation*}
   \J_{\pi}=\begin{pmatrix}
       -\pi_1\pi_2^{-1} & \pi_1\pi_2^{-1}\pi_1+\pi_2\\
       -\pi_2^{-1} & \pi_2^{-1}\pi_1
   \end{pmatrix} .
\end{equation*}
In particular, the bivector $\sigma=\pi_1\pi_2^{-1}\pi_1+\pi_2$ is the Poisson bivector associated to $\J_{\pi}$.
\end{customthm}
We then present a new class of examples of generalized complex structures: Poisson-Nijenhuis structures $(\sigma,N)$ with invertible Poisson bivector $\sigma$ and where $N:TM\to TM$ does not need to be invertible.

\medskip
\textbf{Question 3:} {\em Is there any analog of the Weinstein splitting theorem for complex Poisson structures?}

\medskip
We partially answer this question affirmatively. In order to treat this question, we first introduce local models for complex Poisson structures following \cite{frejlich2017normal}. Given a complex vector bundle $p:E\to M$ equipped with a fiber-wise complex symplectic structure $\sigma$, we say that a complex valued two-form $\widetilde{\sigma}\in \Gamma(\wedge^2 T^*E\otimes\C)$ is an {\em extension} of $\sigma$ if $\widetilde{\sigma}|_{TE|_M}=0$. If, in addition, $M$ is equipped with a complex Poisson structure $\pi_M$, any extension $\widetilde{\sigma}$ provides a {\em local model} for $\pi_M$, that is essentially a complex Dirac structure $L(\widetilde{\sigma})$ encoding information of the local structure of $\pi_M$.
We then review some results concerning normal forms for complex Dirac structures, in this part we provide a normal form for a complex Dirac structure with constant order (see Corollary \ref{normal_cx_dirac_co}).

For a complex Poisson manifold $(M,\pi=\pi_1+i\pi_2)$, we introduce {\em mixed submanifolds}, that is, submanifolds $N\xhookrightarrow{} M$ satisfying
    \begin{equation*}
\pi_1(\Ann TN)\oplus TN=TM|_N\text{ and }\pi_2(\Ann TN)=0.
    \end{equation*}

Finally, our main result is the following:
\begin{customthm}{9.15}
    Let $\pi=\pi_1+i\pi_2$ be a complex Poisson structure with constant order and let $N\xhookrightarrow{\iota}M$ be a mixed submanifold. Choose $\epsilon=X+\xi_1+i\xi_2\in \Gamma(\gr(\pi))$ such that $\epsilon|_N=0$ and $X$ is Euler-like along $N$ with associated tubular neighborhood $\psi: \nu_N\to U\subseteq M$. Then $B+i\omega$ is an extension of the complex symplectic vector bundle $((\nu_N)_{\C}, \Omega_{\C})$, where $
B=\int_{0}^{1}\frac{1}{\tau}\kappa^{*}_{\tau}\psi^*(d\alpha_1)d\tau$ and $\omega=\int_{0}^{1}\frac{1}{\tau}\kappa^{*}_{\tau}\psi^*(d\alpha_2)d\tau$ (here $\kappa_{\tau}$ is the multiplication by the scalar $\tau\in \R$). Furthermore, $\T_{\C}\psi: \T_{\C}\nu_N\to \T_{\C}M$ restricts to an isomorphism of complex Dirac structures
    \begin{equation*}
   L(B+i\omega)= (p^!\gr(\pi_N))^{\omega_1+i\omega_2}\to \gr(\pi)|_U,
\end{equation*}
    where $L(B+i\omega)$ is the local model associated to $B+i\omega$.    
\end{customthm}

When the normal bundle is trivial, i.e., $\nu_N=N\times P$, we obtain a version of the Weinstein splitting theorem along mixed submanifolds (see Corollary \ref{cx_weinstein_splitting}).

\subsection*{Outline}
This article is organized as follows: In Section 2, we begin by recalling some basic facts about complex multivectors, and complex multi-derivations. We also present examples of complex Poisson structures, among which the most important are bi-Hamiltonian structures. In Section 3, we describe the complex Lie algebroid on $\CCM$ associated to a complex Poisson structure $\pi$, along with some of its underlying subalgebroids. In Section 4, we introduce the foliation associated to a complex Poisson structure and define a complex presymplectic structure on each leaf. In Section 5, we compare the previously constructed complex presymplectic foliation with the one associated to the graph of the Poisson bivector viewed as a complex Dirac structure. In Section 6, we introduce the appropriate notions of regularity for complex Poisson bivectors and the class of quasi-real Poisson structures. In Section 7, we introduce the class of quasi-real Dirac structures, the second real Dirac structures associated to a complex Dirac structure, and the complex tangent sums, which together provide a Dirac-geometric description of the complex presymplectic foliation. In Section 8, we study the real index of a complex Poisson bivector and develop a reduction procedure in order to annihilate the real index of a complex Poisson bivector, in the last part, we focus on the case of zero real index. In Section 9, we study the local models associated to complex Poisson bivectors and provide a normal form for complex Poisson bivectors along mixed submanifolds. Finally, in Section 10, we present some concluding remarks.

 In the final part of this article, we include three appendices. In the first part, we recall some basic facts of complex linear algebra and clarify the differential used in the complex setting. In the second appendix, we review some key definitions and properties of complex Dirac structures. Finally, in the third appendix, we provide a brief overview of Euler-like vector fields and normal vector bundles.

\medskip

\textbf{Acknowledgements.} The author acknowledges Pedro Frejlich and Hudson Lima for many fruitful conversations.

\medskip

\textbf{Notation and convention}.
All the manifolds and maps considered in this article are smooth, and all submanifolds are assumed to be embedded. Given a vector bundle $A$ over a smooth manifold $M$, a {\em distribution} is an assignment $m\in M\mapsto R_m\subseteq A_m$, where $R_m$ is a vector subspace of the fiber $A_m$; when any $v\in R_m$ can be extended to a local section of $A$ taking values in $R$, we say that the distribution is {\em smooth}. Each distribution has an associated $\mathbb{N}$-valued function, called its {\em rank}, defined by the assignment $\rk: m\in M\mapsto\dim R_m$. A distribution is said to be {\em regular} if it is a vector bundle. 

Given a distribution $E\subset \TCM$, where $A_{\C}$ is the complexification of a real vector bundle $A\to M$, we denote the space of real elements of $E$ by 
\begin{equation*}
    \rea E:=E\cap A
\end{equation*}
 and refer to it as the {\em real part of $E$}.
In the special case of $A=TM$, we denote it by:
\begin{equation*}
    \Delta(E):=E\cap A
\end{equation*} (or just $\Delta$ depending of the context). We denote the projection to $TM$ as 
\begin{equation*}
    D(E)=\rea(E+\overline{E})=\pr_A E
\end{equation*}
 (or just $D$ depending of the context). Given a smooth map $\varphi:M\to N$, we denote by $T_{\C}\varphi$ the complexification of $T\varphi$, i.e., $T\varphi\otimes\C$.

The space of $\C$-valued differential forms, denoted by $\Omega^k(M,\C)$, is defined as
\begin{equation*}
  \Omega^k(M,\C)=\{\omega: TM\times\ldots\times T M\to \C\st\omega\text{ is $\R$-multilinear of degree $k$ }\}  
\end{equation*}
While the space of complex forms $\Omega_{\C}^k(M)$ is given by
\begin{equation*}
  \Omega_{\C}^k(M)=\{\omega: \TCM\times\ldots\times\TCM\to \C\st\omega\text{ is $\C$-multilinear of degree $k$ }\}. 
\end{equation*}
\section{Complex Poisson bivectors}
\subsection{Complex multivectors}
Given a $\C$-multilinear map
\begin{equation*}
\Lambda:\underbracket{C^{\infty}(M,\C)\times\cdots \times C^{\infty}(M,\C)}_{q-times}\to C^{\infty}(M,\C),
\end{equation*}
we say that it satisfies the {\em Leibniz identity} if
\begin{equation*}
\Lambda(fg, f_2,\cdots f_q)=f\Lambda(g, f_2,\cdots f_q)+g\Lambda(f, f_2,\cdots f_q),    
\end{equation*}
for all $f,g,f_2,\ldots f_q\in \CMC$.

Applying the same reasoning as in Lemma \ref{bivdecom}: we can write $\Lambda=\widetilde{\Lambda}_1+i\widetilde{\Lambda}_2$, where 
\begin{equation*}
 \widetilde{\Lambda}_1, \:\widetilde{\Lambda}_2:\underbracket{C^{\infty}(M,\C)\times\cdots \times C^{\infty}(M,\C)}_{q-\text{times}}\to C^{\infty}(M,\R),
\end{equation*}
are $\R$-multilinear. Thus, we consider the real brackets corresponding to $\widetilde{\Lambda_1}, \widetilde{\Lambda_2}$, a straightforward computation shows that
\begin{equation*}\label{realdecompositionbracket}
\Lambda_1=\widetilde{\Lambda}_1|_{C^{\infty}(M,\R)\times\cdots \times C^{\infty}(M,\R)}\:\:\text{and}\:\: \Lambda_2=\widetilde{\Lambda}_2|_{C^{\infty}(M,\R)\times\cdots \times C^{\infty}(M,\R)}.
\end{equation*}
\begin{proposition}
    The operator $\Lambda$ satisfies the Leibniz identity if and only if both $\Lambda_1$ and $\Lambda_2$ satisfy the Leibniz identity.
\end{proposition}
\begin{lemma}\label{opmultivector}
A $\C$-multilinear map 
\begin{equation*}
    \Lambda:\underbracket{C^{\infty}(M,\C)\times\cdots \times C^{\infty}(M,\C)}_{q-\text{times}}\to C^{\infty}(M,\C),
\end{equation*}

arises from a complex multivector field if and only if it is skew-symmetric and satisfies the Leibniz identity.
\end{lemma}
\begin{proof}
Decompose the operator into real and imaginary components, and then apply the standard real case. 
\end{proof}
\begin{definition}
    Let $f=f_1+if_2\in \CMC$. The {\em complexified diferential of $f$} is the $\C$-linear map $T_{\C}f\in \Gamma(T^*_{\C}M)$ defined by:
    \begin{align*}
    T_{\C}f:\TCM&\to \C\\
        T_{\C}f(X+iY)&=Tf_1(X)-Tf_2(Y)+i(Tf_2(X)+Tf_1(Y)).
    \end{align*}
\end{definition}
\begin{remark}
    This differential is defined exclusively for $\C$-valued maps and differs from the one described in the Notation and Convention section, where the maps were reals. Nevertheless, if $f\in \CM\subseteq \CMC$, then the differential $T_{\C}f$, as in the definition above, coincides with the complexification $Tf\otimes\C$. 
\end{remark}
The passage from a complex multivector field $\Pi\in \Gamma(\wedge^k TM_{\C})$ to its associated operator and vice-versa is as follows: given $f_1,\ldots, f_k\in\CMC$, we define
\begin{equation*}
\Lambda_{\Pi}(f_1,\ldots,f_k)=\Pi(T_{\C}f_1\ldots,T_{\C}f_k).
\end{equation*}

A $\C$-multilinear map 
$\Lambda:\underbracket{C^{\infty}(M,\C)\times\cdots \times C^{\infty}(M,\C)}_{q-\text{times}}\to C^{\infty}(M,\C),$ satisfying the Leibniz identity has associated ot it two $\R$-linear operators $\Lambda_1$ and $\Lambda_2$, which also satisfy the Leibniz identity. Therefore, $\Lambda_1$ and $\Lambda_2$ define real multivector fields $\Pi_1, \Pi_2$, respectively, satisfying the following:

\begin{lemma}\label{cxmultivector}
The complex multivector field $\:\Pi$ satisfies $\Pi=\Pi_1+i\Pi_2$, in the sense of Lemma \ref{multicxisom}.
\end{lemma} 
\subsection{Complex Poisson structures}
Now we introduce the main object of this article:
\begin{definition}
A {\em complex Poisson structure} is a Lie bracket on $C^{\infty}(M,\C)$:
\begin{equation*}
\{\cdot,\cdot\}:C^{\infty}(M,\C)\times C^{\infty}(M,\C)\to C^{\infty}(M,\C)
\end{equation*}
satisfying the Leibniz identity
\begin{equation*}
\{f, gh\}=g\{f,h\}+h\{f,g\},
\end{equation*}
for all $f,g,h\in \CMC$.
\end{definition}
Briefly, we can say that $(\CMC, \{\cdot,\cdot\})$ is a Poisson algebra over $\C$. Morphisms between complex Poisson structures are defined as morphisms between Poisson algebras.

By equation \eqref{realdecompositionbracket}, the bracket admits the following the decomposition:
\begin{equation*}
    \{\cdot,\cdot\}= \{\cdot,\cdot\}_1+i \{\cdot,\cdot\}_2,
\end{equation*}
where $\{\cdot,\cdot\}_1,\{\cdot,\cdot\}_2:\CM\times\CM\to\CM$ satisfy the Leibniz identity.

On the other side, an immediate consequence of Lemma \ref{cxmultivector} is that a Poisson bracket on $\CMC$ is equivalent to a complex bivector $\pi\in \Gamma(\wedge^2 TM_{\C})$ satisfying the identity $[\pi,\pi]=0$. Since the brackets $\{\cdot,\cdot\}_1$ and $\{\cdot,\cdot\}_2$ both satisfy the Leibniz identity, they define real bivectors $\pi_1$ and $\pi_2$, respectively, and we obtain the decomposition $\pi=\pi_1+i\pi_2$.

Using the properties of the complex Schouten bracket (see Appendix \ref{apcxcartan}) we observe that the condition
\begin{equation*}
[\pi_1+i\pi_2,\pi_1+i\pi_2]=0
\end{equation*} is equivalent to the conditions
\begin{equation*}
[\pi_1,\pi_2]=0\text{ and }[\pi_1,\pi_1]=[\pi_2,\pi_2].
\end{equation*} 
Perceive that $\pi_1$ and $\pi_2$ are not necessarily Poisson structures themselves, however, they satisfy the following compatibility condition:
\begin{equation*}
Jac(\{\cdot,\cdot\}_1)=Jac(\{\cdot,\cdot\}_2),
\end{equation*}
where $Jac$ indicates the Jacobiator of the corresponding bracket. 
\subsubsection{Description in local coordinates}
Let $\{\cdot,\cdot\}$ be a complex Poisson bracket with real and imaginary parts $\{\cdot,\cdot\}_1$ and $\{\cdot,\cdot\}_2$, respectively, and let $(x_1,, \ldots, x_n)$ be a local coordinate system on an open set $U$. Consider
\begin{equation*}
    \{x_i,x_j\}_1=\pi^1_{i,j}\:\:\text{and}\:\: \{x_i,x_j\}_2=\pi^2_{i,j}
\end{equation*}
So in local coordinates, the bivector is expressed as
\begin{align*}
\pi&=\{x_k,x_j\}_1\frac{\partial}{\partial x_k}\wedge \frac{\partial}{\partial x_j}+i\{x_k,x_j\}_2\frac{\partial}{\partial x_k}\wedge \frac{\partial}{\partial x_j}\\
&=\pi^1_{kj}\frac{\partial}{\partial x_k}\wedge \frac{\partial}{\partial x_j}+i\pi^2_{kj}\frac{\partial}{\partial x_k}\wedge \frac{\partial}{\partial x_j}.
\end{align*}
The Jacobi condition is equivalent to the following system of partial differential equations:
\begin{equation}\label{pdejac1}
    \sum_{l, s} (-1)^{s+1} (\pi^s_{il}\frac{\partial \pi^s_{jk}}{\partial x_l}+\pi^s_{kl}\frac{\partial \pi^s_{ij}}{\partial x_l}+\pi^s_{jl}\frac{\partial \pi^s_{ki}}{\partial x_l})=0
\end{equation}
\begin{equation}\label{pdejac2}
     \sum_{l} (\pi^2_{il}\frac{\partial \pi^1_{jk}}{\partial x_l}+\pi^2_{kl}\frac{\partial \pi^1_{ij}}{\partial x_l}+\pi^2_{jl}\frac{\partial \pi^1_{ki}}{\partial x_l}+\pi^1_{il}\frac{\partial \pi^2_{jk}}{\partial x_l}+\pi^1_{kl}\frac{\partial \pi^2_{ij}}{\partial x_l}+\pi^1_{jl}\frac{\partial \pi^2_{ki}}{\partial x_l})=0.
\end{equation}
Hence, in order to obtain a complex Poisson bivector we need to solve the system of PDEs given by equations \eqref{pdejac1} and \eqref{pdejac2}. 

\subsection{Examples}\label{examples}
\begin{enumerate}[1)]
 \item {\em Complex symplectic manifolds:} A complex symplectic structure is a closed complex two-form $\omega\in \Gamma(\wedge^2 T^*M_{\C})$, such that $\omega:\TCM\to T_{\C}^*M$ is an isomorphism. The equivalence between differentials, given in Appendix \ref{apcxcartan}, ensures that the bivector $\pi=\omega^{-1}$ satisfies the Jacobi identity if and only if $d\omega=0$. If, in particular $\omega$ is a symplectic structure, then $\omega_{\C}$ is a complex symplectic structure. Decomposing $\omega=\omega_1+i\omega_2$ and $\omega^{-1}=\pi=\pi_1+i\pi_2$, we obtain the following identities from a straightforward computation:
\begin{equation*}
    \pi_1\omega_2+\pi_2\omega_1=0
\end{equation*}
\begin{equation*}
    -\pi_2\omega_2+\pi_1\omega_1=Id.
\end{equation*}
A direct consequence of these identities is that, if $\omega$ is a symplectic structures, then 
\begin{equation*}
(\omega^{-1})_{\C}=(\omega_{\C})^{-1}.
\end{equation*}
 \item {\em Conjugate:} let $\pi=\pi_1+i\pi_2$ be a complex Poisson bivector, its {\em conjugate} is defined as the complex Poisson bivector 
    \begin{equation*}
    \pi^c=\pi_1-i\pi_2.
    \end{equation*}
    \item {\em Bi-Hamiltonian:} A bi-Hamiltonian system  $(\pi_1,\pi_2)$ is a pair of compatible Poisson bivectors, i.e.,  two Poisson bivectors $\pi_1,\pi_2$ satisfying the compatibility condition $[\pi_1,\pi_2]=0$. In this case, the complex bivector $\pi=\pi_1+i\pi_2$ is a complex Poisson bivector. For example, the bivector
    \begin{equation*}
        \pi=x_3 \frac{\partial}{\partial x_1}\wedge \frac{\partial}{\partial x_2}+ix_1 \frac{\partial}{\partial x_2}\wedge \frac{\partial}{\partial x_3}
    \end{equation*}
arises from a bi-Hamiltonian structure. Yet, we can link another complex Poisson structure to a bi-Hamiltonian structure by taken
\begin{equation*}
    i\pi^c=\pi_2+i\pi_1.
\end{equation*}
    We shall see later that the bivectors $\pi$ and $i\pi^c$ are not very similar.
    \item {\em Holomorphic Poisson:} A holomorphic Poisson bivector over a holomorphic manifold $(M,I)$ is a bivector $\pi\in \Gamma(\wedge^2 T^{0,1})$, satisfying  $\overline{\partial}\pi=0$ and $[\pi,\pi]=0$. It was proved in \cite{laurent2008holomorphic}, that a holomorphic Poisson bivector is equivalent to a complex Poisson bivector $\sigma=\sigma_1+i\sigma_2$, where $\sigma_1$ and $\sigma_2$ are both real Poisson bivectors, satisfying $\sigma_1=\sigma_2\circ I^*$, and the pair $(\sigma_2, I)$ is a Poisson-Nijenhuis structure. 
        \item {\em Product of complex Poisson structures:} it works exactly in the same way as in the real case.
    \item {\em Complexification of Poisson bivectors:} if $\sigma$ is a Poisson bivector, then the complex bivectors $\pi=\sigma_{\C}$ and $\pi'=i\sigma_{\C}$ are both complex Poisson bivectors, represented as $\pi=\sigma+i0$ and $\pi'=0+i\sigma$, respectively. We shall say that $\pi'$ is the {\em twist complexification} of $\sigma$.
    \item {\em Diagonal complexification:} Given a Poisson bivector, the complex bivector $\pi=\sigma+i\sigma$ is a complex Poisson bivector that we shall call the {\em diagonal complexification of $\sigma$}. More generally, we have the following family of complex Poisson bivectors
     \begin{equation*}
        \pi=\lambda_1\sigma+i\lambda_2\sigma=(\lambda_1+i\lambda_2)\sigma,
    \end{equation*} 
    where $\lambda_1,\lambda_2\in \R$.
    \item {\em Two-parameter family of complex Poisson bivectors:} for $\mu,\lambda\in \C$, consider the bivector 
    \begin{equation*}
    \pi_{\mu,\lambda}=\mu\pi_1+i\lambda\pi_2.
    \end{equation*} 
    Taking real and imaginary parts of the Schouten bracket, we get:
    \begin{equation*} [\pi_{\mu,\lambda}, \pi_{\mu,\lambda}]=(\mu^2-\lambda^2)[\pi_1,\pi_1].
    \end{equation*} As a consequence, if $\mu^2-\lambda^2=0$, then $\pi_{\mu,\lambda}$ is a complex Poisson bivector. Note that, in case $\pi_1$ and $\pi_2$ form a bi-Hamiltonian system, $\pi_{\mu,\lambda}$ is Poisson for every $\mu,\lambda\in \C$.
    \item Consider $M=\R^3$ with coordinates $(x,y,z)$ and $a, b\in \R$, with $b\neq 0$. We define the bracket
    \begin{equation*}
\{\cdot,\cdot\}:C^{\infty}(\R^3,\C)\times C^{\infty}(\R^3,\C)\to C^{\infty}(\R^3,\C)
\end{equation*}
    \begin{equation}\label{nonbiham}
        \{x,y\}=1+ia, \:\: \{x,z\}=ib,\:\:\{y,z\}=y+i(-ay+(\frac{1+a^2}{b})z).
    \end{equation}
    Note that:
    \begin{align*}
        \{x,\{y,z\}\}&=\{x, y+i(-ay+(\frac{1+a^2}{b})z)\}\\
        &=1+ia+i(-a\{x,y\}+(\frac{1+a^2}{b})\{x,z\})\\
        &=1+ia+(-a(1+ia)+(\frac{1+a^2}{b})ib)\\
        &=0.
    \end{align*}
   Since $a$ and $b$ are constant, we have
    \begin{equation*}
         \{z,\{x,y\}\}=\{z, 1+ia\}=0\text{ and } \{y,\{z,x\}\}=-i\{y, b\}=0.
    \end{equation*}
   
    Hence, the bracket $\{\cdot,\cdot\}$ satisfies Jacobi and thus it is a complex Poisson bracket with associated bivector
    
    \begin{equation*}
        \pi=\px\wedge\py+y\py\wedge\pz+i(a\px\wedge\py+b\px\wedge\pz+(-ay+(\frac{1+a^2}{b})z)\py\wedge\pz).
    \end{equation*}
\end{enumerate}
The real and imaginary components of $\{\cdot,\cdot\}$ are given by:
\begin{equation*}
\{\cdot,\cdot\}_1, \{\cdot,\cdot\}_2,:C^{\infty}(\R^3)\times C^{\infty}(\R^3)\to C^{\infty}(\R^3)
\end{equation*}
\begin{equation*}
        \{x,y\}_1=1, \:\: \{x,z\}_1=0,\:\:\{y,z\}_1=y.
    \end{equation*}
    and
    \begin{equation*}
        \{x,y\}_2=a, \:\: \{x,z\}_2=b,\:\:\{y,z\}_2=-ay+\frac{1+a^2}{b}z.
    \end{equation*}
    Neither bracket satisfies the Jacobi identity. Therefore, $\{\cdot,\cdot\}_1$ and $\{\cdot,\cdot\}_2$ do not come from a bi-Hamiltonian.
\subsection{Hamiltonian vector fields}\label{ham_vc_fields}
By Proposition \ref{opmultivector}, $\C$-linear operator $X:\CMC\to \CMC$ satisfying the Leibniz identity corresponds to a complex vector field. Then, we decompose $X$ into two real vector fields $X_1,X_2:\CMR\to \CMR$ such that 
\begin{equation*}
X(f_1+if_2)= X_1(f_1)-X_2(f_2)+i(X_1(f_2)+X_2(f_1)).   
\end{equation*}
 In term of bivector, we have the decomposition $X=X_1+iX_2$.
Let $\{\cdot,\cdot\}$ be a complex Poisson bracket and let $h=h_1+ih_2\in \CMC$. The map 
\begin{align*}
    \{\cdot,h\}: & \CMC\to\CMC\\
   &f_1+if_2\mapsto \{f_1+if_2,h\} 
\end{align*}
is a $\C$-linear map satisfying Leibniz. Thus, it corresponds to a complex vector field $X_h\in \X_{\C}(M)$, given by $X_h(f)=\{f,h\}$, which we call the {\em complex Hamiltonian vector field associated to $h$}.

\medskip
Let $\{\cdot,\cdot\}_1$ and $\{\cdot,\cdot\}_2$ be the factors of the decomposition in real and imaginary part of $\{\cdot,\cdot\}$. For a function $g\in \CMR$, denote by  $X^1_g$ and $X^2_g$ to the Hamiltonian vector fields associated to the brackets $\{\cdot,\cdot\}_1$ and $\{\cdot,\cdot\}_2$, respectively. Let $h\in\CMC$; decompose the hamiltonian associated to $h$ as $X_h=(X_h)_1+i(X_h)_2$, with $(X_h)_1$, $(X_h)_2$ real vector fields. Then, we have the following:
\begin{proposition} Let $h=h_1+h_2\in\CMC$. Decompose $X_h$ in real and imaginary factors: 
\begin{equation*}
    X_h=(X_h)_1+i(X_h)_2.
\end{equation*} Then,
    \begin{align*}
    (X_h)_1(f)=\{f,h_1\}_1-\{f,h_2\}_2=X^1_{h_1}(f)-X^2_{h_2}(f)\\
     (X_h)_2(f)=\{f,h_2\}_1-\{f,h_1\}_2=X^1_{h_2}(f)-X^2_{h_1}(f).
    \end{align*}
\end{proposition}
\begin{proof}
    Straightforward.
\end{proof}

Let $L$ be a complex Dirac structure. A function $h\in\CMC$ is called {\em admissible} if there exists a complex vector field $X\in \Gamma(\TCM)$ such that $X+T_{\C}h\in L$. We denote by $X_h$ to such a vector field and call it the {\em Hamiltonian vector field} associated to $h$ (as in the real setting, there could be many of such vectors). We denote by $C_{adm}^{\infty}(M, \C)$ to the space of admissible functions. There is an obvious choice of bracket on admissible functions
\begin{equation*}
    \{f,g\}=L_{X_{f}}g,
\end{equation*}
where $X_f$ is the Hamiltonian vector field associated to $f$. It is easy to see that this bracket defines a complex Lie algebra structure on $C_{adm}^{\infty}(M,\C)$.
\begin{proposition}
 The space of admissible functions is a complex Lie algebra.
\end{proposition}

\section{Underlying Lie algebroids}\label{assoc_lie_alg}
\begin{definition}\label{def_cla}
A {\em complex Lie algebroid} over a manifold $M$ is a complex vector bundle $A\to M$ equipped with a $\C$-linear bundle map $\rho: A\to \TCM $ called {\em anchor map} together with a Lie bracket on $\Gamma(A)$
\begin{equation*}
[\cdot,\cdot]:\Gamma(A)\times \Gamma(A)\to\Gamma(A)
\end{equation*}
satisfying the Jacobi identity and the Leibniz identity
\begin{equation*}
 [\alpha,f \beta]=f[\alpha,\beta]+\rho(\alpha)(f)\beta,
\end{equation*}
 for all $\alpha,\beta\in \Gamma(A)$ and for any function $f\in C^{\infty}(M,\C)$.
\end{definition}
To any complex Lie algebroid $(A,[\cdot, \cdot],\rho)$, we can associate to it a real distribution of $A$ defined~by
\begin{equation*}
    A^{\rea}=\rho^{-1}(\Delta(\rho(A)).
\end{equation*}
When $A^{\rea}$ is a vector bundle, it is automatically a Lie algebroid, for more details, see \cite{aguero2024complexliealgebroidsconstant}.

For a complex Poisson bivector $\pi=\pi_1+i\pi_2 \in \Gamma(\wedge^{2}T^*_{\C}M)$, the triple $(T^*_{\C}M, [\cdot, \cdot]_{\pi}, \pi^{\shrp})$ defines a complex Lie algebroid that we denote by $(T^*_{\C}M)_{\pi}$. Consider the bracket
\begin{equation*}
    [\cdot, \cdot]_{\pi}: \Omega^1_{\C}(M)\times \Omega^1_{\C}(M)\to \Omega^1_{\C}(M)
\end{equation*}
 defined as in the real Poisson setting: 
\begin{equation}\label{p_bracket}
[\alpha,\beta]_{\pi}=L_{\pi^{\shrp}(\alpha)}\beta-L_{\pi^{\shrp}(\beta)}\alpha-T_{\C}\pi(\alpha,\beta).
\end{equation}
By the identities established in Appendix \ref{apcxcartan}, we have the following equality:
\begin{equation*}
    [T_{\C}f,T_{\C}g]=T_{\C}\{f,g\},
    \end{equation*}
    for all $f,g\in \CMC$.
The anchor map associated to $(T^*_{\C}M)_{\pi}$ is given by $\pi^{\shrp}:T_{\C}^*M\to T_{\C}M$, which decomposes as:
\begin{equation*}
\pi^{\shrp}=\rho_1+i\rho_2,
\end{equation*}
where $\rho_1, \rho_2:T^*_{\C}M\to TM$ are $\R$-linear. Note that $\rho_1|_{TM}=\pi^{\shrp}_1$ and $\rho_2|_{TM}=\pi^{\shrp}_2$. Thus, the explicit formula for $\pi^{\shrp}$ is:
\begin{equation*}
    \pi^{\shrp}(\xi+i\eta)=\pi^{\shrp}_1(\xi)-\pi^{\shrp}_2(\eta)+i(\pi^{\shrp}_2(\xi)+\pi^{\shrp}_1(\eta)),
\end{equation*}
where $\xi+i\eta\in T^*_{\C}M$. Consequently, the following identifications hold:

\begin{equation*}
    \rho_1(\xi+i\eta)=\pi_1(\xi)-\pi_2(\eta)\:\:\text{and}\:\:\rho_2(\xi+i\eta)=\pi_2(\xi)+\pi_1(\eta).
\end{equation*}

Observe that $\pi_1$ and $\pi_2$ define skew algebroids\footnote{A skew algebroid is a Lie algebroid where the Jacobi identity is not necessarily satisfied} given by $(T^*_{\pi_1}M,[\cdot, \cdot]_{\pi_1}, \pi_1^{\shrp})$ and $(T^*_{\pi_2}M,[\cdot, \cdot]_{\pi_2}, \pi_2^{\shrp})$. Here, the brackets $[\cdot, \cdot]_{\pi_1}$ and $[\cdot, \cdot]_{\pi_2}$ are defined in the exact same way as in equation \eqref{p_bracket} but replacing complex one-forms with real one-forms. A straightforward verification shows that for any real one-forms $\alpha,\beta\in \Omega^1(M)$, we have:
\begin{equation}\label{sumbracket}
    [\alpha,\beta]_{\pi}=[\alpha,\beta]_{\pi_1}+i[\alpha,\beta]_{\pi_2}.
\end{equation}
Applying $\pi^{\shrp}$ to both sides of the equation and using the identity $\pi^{\shrp}[\alpha,\beta]_{\pi}=[\pi^{\shrp}(\alpha),\pi^{\shrp}(\beta)]$, we then compare the real and imaginary parts, we get the following identities:
\begin{lemma}
The following identities holds:
    \begin{equation}\label{1stidentity}
            [\pi_1^{\shrp}(\alpha),\pi_1^{\shrp}(\beta)]-[\pi_2^{\shrp}(\alpha),\pi_2^{\shrp}(\beta)]=\pi_1^{\shrp}[\alpha,\beta]_{\pi_1}-\pi_2^{\shrp}[\alpha,\beta]_{\pi_2}
        \end{equation}
        \begin{equation}\label{2ndidentity}
            [\pi_2^{\shrp}(\alpha),\pi_1^{\shrp}(\beta)]-[\pi_1^{\shrp}(\alpha),\pi_2^{\shrp}(\beta)]=\pi_2^{\shrp}[\alpha,\beta]_{\pi_1}+\pi_1^{\shrp}[\alpha,\beta]_{\pi_2}
        \end{equation}
\end{lemma}
\begin{proposition}\label{kernelrealbivector}
    If $\ker\pi_2\subseteq T^*M$ ($\ker\pi_1\subseteq T^*M$, respectively)  is a subbundle of $T^*M$, then
    \begin{enumerate}[a)]
        \item The triple $(\ker\pi_2,[\cdot,\cdot]_{\pi_1},\pi_1^{\shrp})$ defines a Lie subalgebroid of $T_{\pi_1}^*M$ (the triple $(\ker\pi_1,[\cdot,\cdot]_{\pi_2},\pi_2^{\shrp})$ defines a Lie subalgebroid of $T_{\pi_2}^*M$, respectively).
        \item The smooth distribution $\pi_1^{\shrp}(\ker\pi_2)$ is integrable (the smooth distribution $\pi_2^{\shrp}(\ker\pi_1)$ is integrable, respectively).
        \item The subbundle $(\ker\pi_2)_{\C}$ is a complex Lie subalgebroid of $(\CCM)_{\pi}$ (the subbundle $(\ker\pi_1)_{\C}$ is a complex Lie subalgebroid of $(\CCM)_{\pi}$, respectively).
    \end{enumerate}
\end{proposition}
\begin{proof}
We prove the proposition for $\ker\pi_2$; for $\ker\pi_1$ the proof is step by step the same. 

a) Note that if $\alpha,\beta\in \ker\pi_2$, then $[\alpha,\beta]_{\pi_2}=0$. By equation \eqref{2ndidentity}, $\pi_2^{\shrp}[\alpha,\beta]_{\pi_1}=0$, so $\ker\pi_2$ is closed under the bracket $[\cdot,\cdot]_{\pi_1}$. By equation \eqref{sumbracket}, we have the identity $[\alpha,\beta]_{\pi}=[\alpha,\beta]_{\pi_2}$ and the fact that $[\cdot,\cdot]_{\pi}$ satisfies the Leibniz identity implies that $[\cdot,\cdot]_{\pi_1}$ satisfies the Jacobi identity when restricted to $\ker\pi_2$. The Leibniz identity for $[\cdot,\cdot]_{\pi_1}$ is already satisfied, since $T^*_{\pi_1}M$ is a skew-algebroid.

Items b) and c) follow directly from a).
\end{proof}
\begin{corollary}
    If $\pi$ is complex Poisson bivector and $m\in M$, then \begin{equation*}
    \ker\pi^{\shrp}|_m=\ker\pi^{\shrp}_1|_m\cap\ker\pi^{\shrp}_2|_m
    \end{equation*}and it is a complex Lie algebra.
\end{corollary}
    The vector spaces $\ker\pi^{\shrp}_1|_m$ and $\ker\pi^{\shrp}_2|_m$ are not necessarily real Lie algebras (unless $\pi_1$ and $\pi_2$ are Poisson bivectors); however both carry brackets that may fail to satisfy the Jacobi identity. It is notable that their intersection becomes a complex Lie algebra.
\begin{corollary}
    Let $\pi$ be a complex Poisson bivector and let $m\in M$ be a point such that $\pi|_m=0$. Then, $\CCM|_m$ inherits a complex Lie algebra structure from $\pi$.
\end{corollary}

\medskip
The Poisson bivector has associated the following distributions:
\begin{equation}\label{assoc_dist}
    E_{\pi}=\pi^{\shrp}(\TCM),\:\:\Dpi=\pr_{TM}E\:\:\text{and}\:\:\deltpi=E_{\pi}\cap TM.
\end{equation}
        
As a complex Lie algebroid, $(\CCM)_{\pi}$ has associated the following distributions
\begin{equation*}
    A_{\pi}=\pi^{-1}(\deltpi)\:\:\text{and}\:\: A^{\min}_{\pi}= A_{\pi}+i A_{\pi},
\end{equation*}
while $A_{\pi}$ is only a real distribution of $(\CCM)_{\R}$, the larger $A^{\min}_{\pi}$ is a complex distribution of $\CCM$. Note that $A_{\pi}=(\CCM)^{\rea}_{\pi}$, thus, if  $A_{\pi}$ is a real vector subbundle of $(T_{\C}^*M)_{\R}$, then it is itself a real Lie algebroid. Since $A_{\pi}\subseteq (T_{\C}^*M)_{\R}\cong ((T_{\C}M)_{\R})^*$ there exists a real subbundle $R$ of $(\TCM)_{\R}$ such that $A_{\pi}=\Ann R$. 
So we have the following:
\begin{proposition}
    \begin{equation*}
        A_{\pi}=\Ann \pi^{\shrp}(\Ann \deltpi)_{\R}.
    \end{equation*}
\end{proposition}
\begin{proof}
Due to the skew-symmetry of $\pi$, we have $R=\pi(\Ann \Delta)_{\R}$ and, consequently, $A_{\pi}=\Ann \pi^{\shrp}(\Ann \Delta)_{\R}$.

\end{proof}
\section{Associated presymplectic foliation}
\subsection{Complex presymplectic manifolds}
A {\em complex presymplectic form} is a complex two-form $\omega\in \Gamma(\wedge^2 T^*_{\C}M)$ such that $d\omega=0$. By Lemma \ref{bivdecom}, there exists two presymplectic two-forms $\omega_1$ and $\omega_2$ such that $\omega=\omega_1+i\omega_2$. Note that 
\begin{equation*}
\ker\omega =\{X+iY\st \omega_1(X)-\omega_2(Y)=0,\:\: \omega_2(X)+\omega_1(Y)=0\}\subseteq T_{\C}M
\end{equation*}
is a not-necessarily-regular involutive distribution and that its real part is 
\begin{equation*} 
\Delta(\ker{\omega})=\ker\omega_1\cap \ker\omega_2.    
\end{equation*}
 
From the perspective of Dirac geometry, $L_{\omega}=\gr\omega$ is a complex Dirac structure. A straightforward computation show the following:
\begin{lemma} For a complex two-form $\omega\in \Gamma(\wedge^2 T^*_{\C}M)$, the following holds:
    \begin{equation*}
        \rea L_{\omega}=\gr({\omega^{\bemol}_1|_{\ker \omega_2}}).
    \end{equation*}
\end{lemma}
Any complex Dirac structure $L$ has associated a real lagrangian family $\widehat{L}$, which is a Dirac structure when $L$ has constant order, see Definition \ref{ri_order}. It is easy to see that the Dirac structure associated to $L_{\omega}$ is $\widehat{L_{\omega}}=L_{\omega_2}$. Indeed,
\begin{equation*}
\widehat{L_{\omega}}=\frac{1}{2i}\cdot( L_{\omega}\ast(-1)\cdot\overline{L_{\omega}})\cap \T M=\frac{1}{2i}\cdot(L_{\omega}\ast L_{-\overline{\omega}})\cap \T M=L_{\omega_2}.\end{equation*}

Finally, consider a submanifold $N\xhookrightarrow{\iota} M$. Then, $\iota^*\omega=\iota^*\omega_1+i\iota^*\omega_2$ is a complex presymplectic structure on $N$.
\subsubsection{Complex symplectic manifolds}
Let $(M,\omega)$ be a complex symplectic manifold, $N\xhookrightarrow{\iota}M$ a submanifold and $E$ a complex distribution of $\TCM|_N$. As in the real case, we have 
\begin{equation*}
E^{\omega}=\{X\in \TCM\st \omega(X, Z)=0,\:\forall Z\in E\}. 
\end{equation*} So we have as in the real setting, notions as isotropic subbundles ($E\subseteq E^{\omega}$), coisotropic subbundles ($E^{\omega}\subseteq E$) and lagrangian subbundles ($E= E^{\omega}$). As well as the correspondings classes of coisotropic, isotropic and lagrangian submanifolds. 

In general, there is no much information about complex symplectic structures, it is not even clear what their local description would be.
\begin{example}
  Consider the cotangent bundle $T^*M$ with coordinates $(x_1,\cdots, x_n, \xi_1\cdots,\xi_n)$, where the ${x_k} 's$ correspond to the manifold coordinates and ${\xi_k}'s$ to the fiber coordinates. Let $\{k_1,\ldots, k_n\}$ and $\{r_1,\ldots, r_n\}$ be families of real numbers. Define the complex two-form:
  \begin{equation*}
      \omega_{k_1,\ldots,k_n,r_1,\ldots,r_n}=\sum^n_{j=1} (k_j dx_j\wedge d\xi_j+ir_j dx_j\wedge d\xi_j).
  \end{equation*}
  The two-form $\omega_{k_1,\ldots,k_n,r_1,\ldots,r_n}$ is closed and a straightforward computation shows that it is non-degenerate if and only if, for each $j$, we have $k_j\neq 0$ or $r_j\neq 0$. Moreover, $\omega_{k_1,\ldots,k_n,r_1,\ldots,r_n}$ is exact. Indeed, consider the complex one-form: 
  \begin{equation*}
       \alpha_{k_1,\ldots,k_n,r_1,\ldots,r_n}=\sum^n_{j=1} (k_j \xi_j dx_j+ir_j\xi_j dx_j),
  \end{equation*}
  for which we have:
  \begin{equation*}
  \omega_{k_1,\ldots,k_n,r_1,\ldots,r_n}=-d \alpha_{k_1,\ldots,k_n,r_1,\ldots,r_n}. 
  \end{equation*}
\end{example} 
The normal form for complex symplectic structures remains still true at the linear algebra level. However, studies such as \cite{turiel1994classification} suggest that simultaneous local or normal form for pairs of compatible two-forms on a smooth manifolds does not necessarily yield both two-forms in the standard local form simultaneously.
\subsection{Associated foliation}
Consider a complex Poisson bivector $\pi=\pi_1+i\pi_2 \in \Gamma(\wedge^{2}T_{\C}M)$. 
Let $E=\pi^{\shrp}(T^*_{\C}M)$ be the image of $\pi^{\shrp}$. Note that $E$ is a complex distribution in $\TCM$; as in the real Poisson case, we define the point-wise complex two-form
\begin{align*}
       \widetilde{\Omega}_p: E|_p\times E|_p&\to \C\\
                         (\pi^{\shrp}(\xi+i\eta), \pi^{\shrp}(\xi'+i\eta'))&\mapsto \pi(\xi+i\eta, \xi'+i\eta'),           
\end{align*}
where $p\in M$ and $\xi+i\eta, \xi'+i\eta'\in \CCM|_p$. Due to the skew-symmetry of $\pi$, $\widetilde{\Omega}_p$ is well-defined. Moreover it is easy to see that $\widetilde{\Omega}_p$ is non-degenerate as a complex two-form. Summarizing:
\begin{proposition}\label{cxsymplecticfamily}
    A complex Poisson bivector defines a family of complex symplectic vector spaces $\{(E|_p, \widetilde{\Omega}_p)\}_{p\in M}$.
\end{proposition}
 Since complex distributions does not have associated a foliation, we consider the real distribution 
 \begin{equation*}
 \Delta_{\pi}=\pi^{\shrp}(T^*_{\C}M)\cap TM.
 \end{equation*}
 In general, $\Delta_{\pi}$ is not an integrable distribution. However, if we assume that $A_{\pi}=\pi^{-1}(\deltpi)$ is a vector bundle, then $A_{\pi}$ is a Lie algebroid with space of orbits $\Delta_{\pi}$. Hence, in this case $\deltpi$ is integrable. 
 
 From now on, we assume that $\deltpi$ is integrable without making any assumption about $A_{\pi}$. Let $S$ be a leaf of $ \deltpi$ and let $p\in S$ be a point. For any $\tau, \tau'\in T_p S$, there exists $\xi+i\eta,\:\: \xi'+i\eta'\in A_{\pi}|_p$ such that $\tau=\rho_1(\xi+i\eta)$ and $\tau'=\rho_1(\xi'+i\eta')$. Then define 
\begin{align*}
&\Omega_S: TS\times TS\to \C,\\
&\Omega_S(\tau,\tau')=\Omega_S(\rho_1(\xi+i\eta),\rho_1(\xi'+i\eta'))=\pi(\xi+i\eta,\xi'+i\eta'),
\end{align*}
a straightforward computation shows that $\Omega_S$ is well-defined.

Summarizing:
\begin{proposition}\label{cx_presymp_leaf}
    If $S$ is a leaf of $\deltpi$, then it inherits a complex presymplectic two-form from~$\pi$.
\end{proposition}

We now calculate the explicit formula of $\Omega_S=\omega_{\text{re}}+i\omega_{\text{im}}$. By equation \eqref{bivdecom}, we have the following general formula for $\pi$:
\begin{equation*}
\pi(\xi+i\eta,\xi'+i\eta')=\pi_1(\xi,\xi')-\pi_1(\eta,\eta')-\pi_2(\xi,\eta')-\pi_2(\eta,\xi')+i(\pi_1(\eta,\xi')+\pi_1(\xi,\eta')+\pi_2(\xi,\xi')-\pi_2(\eta,\eta'))    
\end{equation*}
In our particular case, we have that 
\begin{align*}
\Omega_S(\tau,\tau')&=(\xi'+i\eta')(\pi(\xi+i\eta))\\
&=(\xi'+i\eta')(\pi_1(\xi)-\pi_2(\eta))\\
&=\pi_1(\xi,\xi')-\pi_2(\eta,\xi')+i(\pi_1(\xi,\eta')-\pi_2(\eta,\eta')).
\end{align*}
where $\tau=\rho_1(\xi+i\eta),\tau'=\rho_1(\xi'+i\eta')$. We summarize this in the following proposition:
\begin{proposition}
    Let $\pi=\pi_1+i\pi_2$ be a complex Poisson bivector and let $S$ be a leaf of $\deltpi$. Then, the family of complex symplectic vector spaces $(E, \widetilde{\Omega})$ restricts to $S$ to the complex presymplectic form $\Omega_S\in \Omega^2(S,\C)$, given by:
    \begin{align*}
        \Omega_S(\tau, \tau')&=\omega_{\rea}+i\omega_{\ima}\\ 
       \omega_{\rea}(\tau,\tau')&=\pi_1(\xi,\xi')-\pi_2(\eta,\xi')=\pi_1(\xi,\xi')+\pi_1(\eta,\eta')\\
\omega_{\ima}(\tau,\tau')&=\pi_1(\xi,\eta')-\pi_2(\eta,\eta')=-\pi_2(\xi,\xi')-\pi_2(\eta,\eta'),
    \end{align*}
    where  $\tau=\rho_1(\xi+i\eta)$ and $\tau'=\rho_1(\xi'+i\eta')$.
\end{proposition}

Using Dirac geometry, we shall show in the next section that $\Omega$, $\omega_{\rea}$, $\omega_{\ima}$ are smooth. Furthermore, we shall compute the kernel of $\omega_{\rea}$ and $\omega_{\ima}$.

\medskip
Since we now have a complex presymplectic foliation, a natural question arises: is there a parallel of Casimir functions in the complex Poisson setting?
\begin{definition}
 Let $\{\cdot,\cdot\}$ be a Poisson bracket. A {\em Casimir function} is a function $c\in \CMC$ such that 
    \begin{equation*}
    \{c, h\}=0, \:\:\forall h\in \CMC.
    \end{equation*}
\end{definition}
As in the real Poisson setting, Casimir functions are constants along the complex presymplectic leaves. Indeed, consider a Casimir function $c=c_1+ic_2$. The condition $\{c, h\}=0, \:\:\forall h\in \CMC$ is equivalent to $\pi(T_{\C}c)=0$, which implies that $T_{\C}c|_{\ima \pi}=0$. Since the distribution of real elements $\Delta_{\pi}$ is inside $\ima\pi$. We have that $T_{\C}c|_{\Delta}=0$.  Consequently, both $c_1$ and $c_2$ are constant on the leaves. Hence, we have the following:
\begin{proposition}
    Casimir functions are constant on the leaves of the complex presymplectic foliation associated to $\pi$.
\end{proposition}

\section{Dirac geometry point of view I}
The graph of a complex Poisson bivector $\pi=\pi_1+i\pi_2$ is a lagrangian subbundle of $\TTCM$ that we denote by 
\begin{equation*}
    L_{\pi}=\gr\pi\subseteq \TTCM.
\end{equation*}
It is not difficult to see that the integrability of $\pi$, i.e., the condition $[\pi,\pi]=0$, is equivalent to the involutivity of $L_{\pi}$. Therefore, $L_{\pi}$ is a complex Dirac structure (see Appendix \ref{apcxdirac} for the properties of complex Dirac structures).

\medskip
Now we recall that the integrability condition of a lagrangian subbundle $L$ is equivalent to the vanishing of the three-tensor $T_L$, given by
\begin{align*}
    T_L:L\times L\times L\to \C\\
    T_L(e_1,e_2,e_3)=\la e_1,[e_2, e_3]\ra,
\end{align*}
where $\la\cdot,\cdot\ra$ is the symmetric canonical pairing of $\TTCM$ and $[\cdot,\cdot]$ is the Courant-Dorfmann bracket. It is known that a lagrangian subbundle $L$ is a complex Dirac structure if and only if $T_L=0$.

Let $L$ be a complex Dirac structure and $E=\pr_{TM}L$. The isotropic condition of $L$ allows us to retrieve a well-defined point-wise complex two-form:
\begin{align*}
    \omega_L:E\times E&\to \C\\
    \omega_L(X,Y)&=\eta(Y),
\end{align*}
for $X+\xi, Y+\eta\in L$. The same computation of \cite[Theorem 2.3.5 and 2.3.6]{courant1990dirac}, shows that 
\begin{align*}
\omega_L&=\iota^*_L\la\cdot,\cdot\ra_- \\
    d\omega_L(X_1, X_2, X_3)&=-T_L(e_1,e_2,e_3),
\end{align*}
where $\iota_L$ is the inclusion of the fibers of $L$, $\la\cdot,\cdot\ra_{-}$ is the skew-symmetric canonical pairing of $\T_{\C} M$ and $e_k=X_K+\xi_k\in L$, for $k=1,2,3$.

Let $\Delta_L=\Delta(E)\subseteq TM$ denote the distributions of real elements of $E$. From the discussion above, we have:
\begin{proposition}\label{presympleaf}
    If $S$ is a leaf of $\Delta_L$, then 
    \begin{equation*}
    \omega_S=\iota_S^*\omega_L\in \Omega^2(S,\C)
    \end{equation*}
    is a smooth closed complex two-form.
\end{proposition}
Let $L_{\pi}$ be the graph of a complex Poisson bivector $\pi=\pi_1+i\pi_2$. 
It is easy to see that:
\begin{equation*}
    L_{\pi}=\{\pi_1\xi_1-\pi_2\xi_2+i(\pi_1\xi_2+\pi_2\xi_1)+\xi_1+i\xi_2\st \xi_1+i\xi_2\in \CCM \}.
\end{equation*}
We recall that a complex Dirac structure $L$ has a lagrangian family $\widehat{L}$ associated (see Appendix \ref{apcxdirac}). In our case, applying equation \eqref{lhatid} from Appendix \ref{apcxdirac} directly to $L_{\pi}$, we obtain:
\begin{equation*}
\widehat{L_{\pi}}=\{\pi_1(\xi_1)-\pi_2(\xi_2)+\xi_2\st \pi_2(\xi_1)+\pi_1(\xi_2)=0\}.
\end{equation*}

If $\widehat{L_{\pi}}$ is smooth, then it is automatically a Dirac structure and, consequently, it has an underlying presymplectic foliation $\F$ satisfying $T\F=\Delta_{\pi}$. Let $(S,\omega_S)$ be a presymplectic leaf of $\F$; then a straightforward verification shows the following:

\begin{proposition}\label{same_fol} 
The following identity holds:
\begin{equation*}
\omega_S=-\omega_{\rea}.
\end{equation*}
\end{proposition}

The usual characterization of complex Poisson bivectors extends to the complex Poisson setting: 

Recall that if $S\subseteq \T_{\C}M$, then 
\begin{equation*}
S\cap \TCM=\Ann (\pr_{T^*_{\C}M}S^{\perp}).
\end{equation*} 
Using the previous formula, we have the following:
\begin{lemma}
Let $L\subseteq \T_{\C}M$ be a lagrangian family. Then,
\begin{enumerate}
    \item $L$ is the graph of a complex bivector if and only if $L\cap \TCM=0$.
    \item $L$ is the graph of a complex two-form if and only if $L\cap T^*_{\C}M=0$.
\end{enumerate}
\end{lemma}

\begin{proposition}\label{poisson_in_a_point}
 Let $L$ be a lagrangian family and let $m\in M$ be a point. If $L|_m=T^*_{\C}M|_m$, then there exists a neighborhood $U$ of $m$ and a complex bivector $\pi$ such that $L|_U=\gr(\pi)$. Furthermore, if $L$ is a complex Dirac structure, then $\pi$ is a complex Poisson bivector.
\end{proposition}
\begin{proof}
    Let $\{X_1+\eta_1, \ldots, X_n+\eta_n\}$ be a frame for $L$. Since $L|_m=T^*_{\C}M$, we have $X_1|_m=\cdots=X_n|_m=0$ and thus, $\{\eta_1|_m, \ldots\eta_n|_m\}$ is a linearly independent set. Being linearly independent is an open condition, then there exists a neighborhood $U$ of $m$, where $\{\eta_1, \ldots\eta_n\}$ is linearly independent and so forms a basis of $T^*_{\C}M|_U$. Now define the bundle map $\pi^{\shrp}: \CCM|_U\to \TCM|_U$ with the basis $\{\eta_1, \ldots\eta_n\}$ by setting $\pi(\eta_j)=X_j$ and then extend it linearly to a complex bivector $\pi$; observe that $\gr(\pi)=L|_U$. Moreover, if $L$ is involutive, then $\pi$ is automatically a complex Poisson bivector and $\pi|_m=0$.
\end{proof}
\begin{corollary}\label{open_condition}
   Let $L$ be a complex Dirac structure and let $m\in M$ be a point. If $L|_m=\gr(\pi)$, for a certain bivector $\pi$, then in a neighborhood $U$ of $m$, there exists a complex Poisson bivector $\gamma$ such that $L|_U=\gr(\gamma)$.
\end{corollary}
\begin{proof}
   Note that $e^{-\pi}L|_m=T_{\C}^*M$. Hence, by Proposition \ref{poisson_in_a_point}, there exists a complex bivector $\sigma$ such that $e^{-\pi}L=\gr(\sigma)$ and so $L=\gr(\pi+\sigma)$. Since $L$ is involutive, $\gamma=\pi+\sigma$ is a complex Poisson bivector.
\end{proof}

\section{Regularity on complex Poisson bivectors}
\subsection{Regular and strongly regular complex bivectors}
We review some of the distributions that are associated to a complex Poisson bivector $\pi=\pi_1+i\pi_2$:
\begin{align*}
    E_{\pi}&=\ima\pi^{\shrp}=\{\pi_1(\xi_1)-\pi_2 (\xi_2)+i(\pi_1(\xi_2)+\pi_2 (\xi_1))\st\xi_1, \xi_2\in T^*M\}\subseteq \TCM\\
    \deltpi&=E_{\pi}\cap TM=\{\pi_1(\xi_1)-\pi_2 (\xi_2)\st \pi_1(\xi_2)+\pi_2 (\xi_1)=0,\:\:\xi_1, \xi_2\in T^*M\}\subseteq TM\\
    \Dpi&=\pr_{TM}E_{\pi}=\{\pi_1(\xi_1)-\pi_2 (\xi_2)\st\xi_1, \xi_2\in T^*M\}\subseteq TM.
\end{align*}

\begin{definition}
Let $\pi$ be a complex Poisson bivector. We say that $\pi$ is {\em regular} if $E_{\pi}$ is a regular complex distribution. And we say that $\pi$ is {\em strongly regular} if $\pi$ is regular and if $\deltpi$ is a regular distribution. 
\end{definition}
Another invariant that we can assign to a complex bivector is the $\mathbb{N}$-valued function called {\em order}:
\begin{equation*}
    \order \pi=\rk \Dpi.
\end{equation*}
We say that $\pi$ has {\em constant order} if $\Dpi$ has constant rank.
\begin{example}\label{ex_non_biham}
    Consider the complex Poisson bivector of Examples \ref{examples} equation \eqref{nonbiham} with $a=0$ and $b=1$ over $M=\R^3$:
    
    \begin{equation*}
        \pi=\px\wedge\py+y\py\wedge\pz+i(\px\wedge\pz+z\py\wedge\pz).
    \end{equation*}
    Note that
    \begin{equation*}
    E_{\pi}=span_{\C}\la \py+i\pz, -\px+y\pz+iz\pz\ra, 
\end{equation*}
   so $\dim E_{\pi}=2$. Hence, $\pi$ is regular.
    
A straightforward computation shows that:
    \begin{align*}
        \deltpi&=span\la -\px-z\py+y\pz\ra\\
        \Dpi&=span\la \py, -\px+y\pz, -y\py\ra
    \end{align*}
    Since $\dim \deltpi=1$, $\pi$ is strongly regular. 
\end{example}

\subsection{Quasi-real Poisson structures}
In \cite{aguero2024complexliealgebroidsconstant} the notion of minimal Lie algebroids was introduced. We extend this definition to the setting of complex Poisson structures.
\begin{definition}
We say that complex Poisson bivector $\pi$ is {\em totally real} if it is the complexification of a real Poisson bivector and we say that it is {\em quasi-real} if the image of $\pi^{\shrp}$ is a totally real distribution of $\TCM$.
\end{definition}
Any totally real Poisson bivector is quasi-real. For quasi-real Poisson bivectors we have that 
\begin{equation*}
\pi^{\shrp}(\CCM)={\Dpi}\otimes{\C}=\deltpi\otimes{\C}.
\end{equation*}
We can see this identity as a regularity condition, in the sense that $\Dpi/\deltpi=0$ fiberwise.

\medskip
A {\em minimal Lie algebroid} is a complex Lie algebroid $(A,[\cdot,\cdot],\rho)$ whose anchor image $\rho(A)$ is totally real (see \cite{aguero2024complexliealgebroidsconstant}). For a quasi-real Poisson structure, the Lie algebroid $(\CCM)_{\pi}$ is minimal.

In \cite{aguero2024complexliealgebroidsconstant}, it was proved a normal form for complex Lie algebroids with constant order. Next, we provide a local form for minimal Lie algebroids without relaying on the regularity of the order.
\begin{proposition}
    Let $(A,[\cdot,\cdot],\rho)$ be a minimal Lie algebroid with $D=\rho(A)\cap TM$, $m\in M$. If $\dim D|_m=q$, then there exist a coordinate system $(x_1,\ldots, x_q, y_1,\ldots, y_s)$ and a local frame $\{\alpha_1,\ldots,\alpha_q,\beta_1,\ldots, \beta_r\}$ such that
    \begin{align}
        \rho(\alpha_j)=\pxj,\:\:\text{for}\:\:j=1,\ldots, q \:\:\text{and}\:\:\rho(\beta_j)=\sum b_{j,k}\pyk\:\:\text{for}\:\:j=1,\ldots, r,
    \end{align}
    where the functions $b_{j,k}\in \CMC$ are independents of the $x_k$. Moreover, we have 
    \begin{align}
                [\alpha_i,\alpha_j]&=[\alpha_i,\beta_k]=0, \forall i,j=1,\ldots, q, k=1,\ldots, r\\
        [\beta_i,\beta_j]&=\sum_k \gamma_{i,j}^k\beta_k,\forall i,j=1,\ldots, r,
        \end{align}

    where the functions $\gamma_{i,j}^k\in \CMC$ are independent of the $x_k$.
\end{proposition}
\begin{proof}
    The proof follows the approach in \cite{Dufour2001NormalFF} for real Lie algebroids; however, since we deal with complex vector fields, we perform algebraic manipulations to transform them to real ones. For completeness, we present a proof. 
    
    In the case $q=0$, we have nothing to prove.
    
    Assume $q>0$, and assume that there exists a 
$k<q$ such that there exist a coordinate system $(x_1,\ldots, x_k, y_1,\ldots,y_{s+q-k} )$ and a frame $\{\alpha_1,\ldots,\alpha_k,\beta_1,\ldots\beta_{r+q-k} \}$ satisfying that  
 \begin{align}
 \rho(\alpha_j)&=\pxj,\:\:\text{for}\:\:j=1,\ldots, k\\
 \rho(\beta_j)&\:\:\text{ is independent of the $x_i$ for $j=1,\ldots, r+q-k$, $i=1,\ldots,k$ }\\
  [\alpha_i,\alpha_j]&=[\alpha_i,\beta_u]=0, \forall i,j=1,\ldots, k, u=1,\ldots, r+q-k.
    \end{align}

Notice that $[\beta_u,\beta_v]$ does not have factor of the type $\pxi$, due to the fact that anchor map preserve the bracket and $\rho(\beta_j)$ is independent of the variables $x_i$. So
\begin{equation*}
    [\beta_u,\beta_v]=\sum \gamma_{u,v}^{i}\beta_i, 
\end{equation*}
    where $\gamma_{u,v}^{i}$ are functions on $(x_1,\ldots, x_k, y_1,\ldots,y_{s+q-k} )$. Applying the bracket with $\alpha_i$ and the Leibniz identity, we obtain that
    \begin{equation*}
    [\alpha_i,[\beta_u,\beta_v]]=\sum_j \frac{\partial\gamma_{u,v}^{j}}{\partial x_i} \beta_j.
    \end{equation*}
    
    Using the Jacobi identity and the fact that $\alpha_i$ commute with the $\beta_j$, we obtain that 
    \begin{equation*}
    [\alpha_i,[\beta_u,\beta_v]]=0.
    \end{equation*}
Hence, $\frac{\partial\gamma_{u,v}^{j}}{\partial x_i}=0$, for all $i,u,v$ implying that the complex valued functions $\gamma_{u,v}^{j}$ are independent of the $x_i$.

The proof simplifies if all the $\rho(\beta_j)$ were reals. So assume that there exists a $j\in\{1,\ldots, r+q-k\}$, so that $\rho(\beta_j)|_m$ has non-vanishing real and imaginary parts; for simplicity we take $j=1$. That is
\begin{equation*}
    \rho(\beta_1)=d_1+id_2\in \Gamma(D_{\C}),
\end{equation*} 
where $d_1,d_2\in \Gamma(D)$ and $d_1|_m, d_2|_m\neq 0$.
Since $d_2\in \Gamma(D)$ and $\rho(A)=D_{\C}$, there exist functions $a_j, b_j\in \CMC$ such that 
\begin{equation*}
    \rho(\sum_j a_j\alpha_j+b_j\beta_j)=\sum_j 
 a_j\pxj+b_j\rho(\beta_j)=id_2.
\end{equation*}
 Since $\rho(\beta_1)$ is independent of the $x_j$, the section $d_2$ is too independent. Thus, the functions $a_j=0$ and the functions $b_j$ only depend on the variables $y_j$, for $j=1,\ldots, r+q-k$. Hence, we set
 \begin{equation*}
 \beta'_1=\beta_1-\sum b_j\beta_j.
 \end{equation*}
 Observe that the section $\beta'_1$ satisfies the conditions of the inductive hypothesis. And thus we proceed like with all the other elements with non-vanishing real and imaginary parts.
 
 Now we have that 
 \begin{equation*}
 \rho(\beta_1)|_m=d_1|_m\neq 0,
 \end{equation*}
 then there exist a diffeomorphism only with respect to the coordinates $(y_1,\ldots,y_{r+q-k})$, given by $(t,w_1,\ldots, w_{r+q-k-1})\leftrightarrow (y_1,\ldots,y_{r+q-k})$ such that $\rho(\beta'_1)=\frac{\partial}{\partial t}$.
We know that 
\begin{equation*}
    \rho(\beta_j)=c_j((t,w_1,\ldots, w_{r+q-k-1}))\frac{\partial}{\partial t}+\:\text{ more elements}.
 \end{equation*}
 Then, we take $\beta'_j=\beta_j-c_j\beta'_1$. Note that the frame $\{\beta'_1,\ldots,\beta'_{r+q-k}\}$ still satisfies the conditions of the inductive hypothesis. 

We would like functions $\lambda_{u,v}$ for $u,v=2,\ldots, r+q-k$ dfine on the variables $(t,w_1,\ldots, w_{r+q-k-1})$ satisfying that the sections $\beta''_j=\sum_{u=2}\lambda_{u,j}\beta'_u$ commute with $\beta'_1$, i.e., $[\beta'_1,\beta''_j]=0$. To find this functions is equivalent to solve the following complex EDO system
\begin{equation*}
    \frac{\partial \Lambda}{\partial t}+\Lambda B=0,
\end{equation*}
where $B=[\gamma^i_{1,v}]$.
This system has solution and we take $\alpha_{k+1}=\beta'_1$ and $\beta''_j$ are the new $\beta_j$. 
\end{proof}
\begin{corollary}\label{integrable_image_minimal}
    Let $(A,[\cdot,\cdot],\rho)$ be a minimal complex Lie algebroid. Then $D=\rho(A)\cap TM=\pr_{TM}\rho(A)$ is integrable in the sense of Stefan-Sussmann.
\end{corollary}

In general, a complex Poisson bivector $\pi$ is correlated to a presymplectic foliation (when $\deltpi$ is integrable); however, this foliation does not encode all the information of $\pi$, since we are working only with the real elements of $E=\ima(\pi^{\shrp})$. In contrast, this is not the case for a quasi-real Poisson bivector, now we have that $\Dpi=\deltpi$ defines a foliation $\F$ (without any additional regularity condition) and each leaf of $\F$ inherits a complex symplectic two-form from the complex symplectic family of Proposition \ref{cxsymplecticfamily}. Therefore, we can think of a quasi-real Poisson bivector as a possibly singular complex symplectic foliation. Hence, quasi-real Poisson structures are the most similar objects in the complex Poisson setting to real Poisson bivectors. In summary, we have the following:
\begin{proposition}
    A quasi-real Poisson structure is completely determined by its complex symplectic foliation.
\end{proposition}

\subsection{Poisson-Nijenhuis structures}
We recall that a Poisson-Nijenhuis structure, is given by a pair $(\sigma,N)$, where $\sigma\in \X^2(M)$ is a Poisson bivector and $N:TM\to TM$ is a $(1,1)$-tensor satisfying the following compatibility conditions:
\begin{equation*}
    \sigma^{\shrp}\circ N^*=N\circ\sigma^{\shrp}\:\:\text{ and}
\end{equation*}
\begin{equation*}
    [\alpha,\beta]_{\sigma_N}=[N^*\alpha,\beta]_{\sigma}+[\alpha,N^*\beta]_{\sigma}-N^*[\alpha,\beta]_{\sigma},
\end{equation*}
where $[\cdot,\cdot]_{\sigma}$ and $[\cdot,\cdot]_{\sigma_N}$ are the brackets on one-forms defined by the bivectors $\sigma$ and $\sigma_N=N\circ\sigma$, respectively (see Section \ref{assoc_lie_alg}). We say that a pair $(\omega,N)$, where $\omega$ is a symplectic structure and $N:TM\to TM$ is a $(1,1)-$tensor, is a symplectic-Nijenhuis structure if $(\omega^{-1},N)$ is a Poisson-Nijenhuis structure and $(\omega^{-1})_N$ is symplectic.

It is known that for a Poisson-Nijenhuis structure $(\sigma,N)$, the bivector $\sigma_N$ is also a Poisson bivector. Both bivectors $\sigma$ and $\sigma_N$ have the same underlying foliations but with slightly different symplectic forms. In fact, if $(S, \omega_S)$ is a symplectic leaf of $\sigma$, then $N$ preserves $TS$ and even more the pair $(\omega_S, N|_{TS})$ is a symplectic-Nijenhuis structure. It is easy to verify that the pair $(S, (\omega_S)_N=\omega_S\circ N)$ is a symplectic leaf of $\sigma_N$.
\begin{proposition}\label{pn_qr}
   Let $(\sigma,N)$ be a Poisson-Nijenhuis structure and let $\pi(\sigma,N)=\sigma+i\sigma_N$ be its associated complex Poisson structure. The complex Poisson bivector $\pi(\sigma,N)$ is quasi-real if and only if $N^2(\ima\sigma^{\shrp})=\ima\sigma^{\shrp}$. In particular, this condition holds for maps $N$ satisfying $N^2=\lambda Id$, for some $\lambda\in \R$.
\end{proposition}

\begin{proof}
 Since $N\sigma^{\shrp}=\sigma^{\shrp} N^*$, we have that $\ima\sigma^{\shrp}=\ima \sigma_N^{\shrp}$. Consequently, $E_{\pi}=\ima \pi^{\shrp}\subseteq \ima(\sigma^{\shrp})_{\C}$.

    Recall that 
    \begin{equation*}
    \Dpi=\pr_{TM}E_{\pi}=\{\sigma(\xi_1)-\sigma_N (\xi_2)\st \xi_1, \xi_2\in T^*M\}=\ima\sigma^{\shrp}.
    \end{equation*}
    Moreover,
    \begin{equation*}
        \deltpi=\{\sigma(\xi_1)-\sigma_N (\xi_2)\st \sigma(\xi_2)+\sigma_N (\xi_1)=0\}.
    \end{equation*}
    The condition $\sigma(\xi_2)+\sigma_N (\xi_1)=0$ implies that $N^2\sigma(\xi_1)+N\sigma (\xi_2)=0$. Adding and subtracting the term $\sigma(\xi_1)$, we got that 
    \begin{equation*}
        (N^2+Id)\sigma(\xi_1)=\sigma(\xi_1)-\sigma_N(\xi_2).
    \end{equation*}
    Finally, we obtain that \begin{equation*}
        \deltpi=(N^2+Id)(\ima \sigma^{\shrp}).
    \end{equation*}

    Hence, $\pi$ is a quasi-real Poisson bivector if and only if $N^2+Id$ preserve $\ima\sigma^{\shrp}$. In particular this is true for maps $N$ satisfying that $N^2=\lambda Id$, for $\lambda\neq 0$.
\end{proof}
\begin{example}
    Holomorphic Poisson structures: We recall that a holomorphic Poisson structure over a holomorphic manifold $(M,I)$ can be seen as a Poisson-Nijenhuis pair $(\sigma,-I)$, where $\sigma$ is a Poisson bivector and $I$ is a complex map. Thus, it is equivalent to the complex bivector $\pi=I\circ\sigma+i\sigma$. Since $I^2=-Id$, $\pi$ is a quasi-real Poisson structure.
\end{example}

\begin{example}
    Let $\sigma$ be a Poisson bivector. Then, the pair $(\sigma,Id)$ is trivially Poisson-Nijenhuis, with associated complex Poisson structure $\pi(\sigma,Id)=\sigma+i\sigma$ referred to as the diagonal complexification. Thus, by Proposition \ref{pn_qr}, the diagonal complexification of any Poisson bivector is quasi-real. 
\end{example}

\section{Dirac geometry point of view II}\label{dirac2}
In this section we continue with a Dirac approach to complex Poisson bivectors, our goal is to describe the complex presymplectic foliation in a Dirac geometry fashion (Theorem \ref{cx_presym_fol}).
\subsection{Quasi-real Dirac structures}
\begin{definition}
A {\em quasi-real lagrangian family} is a lagrangian family $L\subseteq \TTCM$ such that $\pr_{TM}L=(D_L)_{\C}$ for certain real distribution $D_L\subseteq TM$.
A {\em quasi-real Dirac structures} is a complex Dirac structure that is also a quasi-real lagrangian family. 
\end{definition}
Unlike what happens with a complex Dirac structure, where the underlying complex presymplectic foliation does not fully determine completely the original complex Dirac structure, for quasi-real Dirac structures, we have the a result analogous to that for quasi-real Poisson structures:
\begin{proposition}
    Every quasi-real Dirac structure has an associated complex presymplectic foliation. Furthermore, a quasi-real Dirac structure is completely determined by its complex presymplectic foliation.
\end{proposition}
\begin{proof}
The proof follows from Corollary \ref{integrable_image_minimal} and Proposition \ref{presympleaf}.
\end{proof}

Quasi-real Dirac structures behave well with respect to backward and forward images: Let $\varphi:M\to N$ be a map. 
\begin{enumerate}[i)]
    \item If $L_M\subseteq \TTCM$ is a quasi-real  Dirac structure with associated real distribution $D_{L_M}$, then $\varphi_{!}L_M$ is a quasi-real Dirac with associated distribution $D_{\varphi_{!}L_M}=T\varphi(D_{L_M})$. 
    
    \item If $L_N\subseteq\T_{\C}N$ is a quasi-real Dirac structure with associated distribution $D_{L_N}$, then $\varphi^{!}L_N$ is a quasi-real Dirac structure with associated distribution $D_{\varphi^{!}L_N}=T\varphi^{-1}(D_{L_N})$.

\end{enumerate}
\subsection{Second associated real Dirac structure}
Every complex Dirac structure $L$ has associated a lagrangian family $\widehat{L}$ defined as the (real) lagrangian family $\widehat{L}$ satisfying:
\begin{equation*}
L\star(-1)\cdot\overline{L}=2i\cdot (\widehat{L})_{\C}    
\end{equation*}
or equivalently
\begin{equation}\label{lhat}
    \widehat{L}=\frac{1}{2i}\cdot(L\star(-1)\cdot\overline{L})\cap \T M
\end{equation}
We can defined another lagrangian family $\widecheck{L}$ as the one satisfying: 
\begin{equation}
L\star\overline{L}=2\cdot (\widecheck{L})_{\C}    
\end{equation}
or equivalently
\begin{equation}\label{lcheck}
    \widecheck{L}=\frac{1}{2}\cdot(L\star\overline{L})\cap \T M.
\end{equation}
\begin{proposition}
    Let $L$ be a lagrangian family. Then, the relation between the operations $\:\:\:\widehat{\cdot}\:\:\:$ and $\:\:\:\widecheck{\cdot}\:\:\:$ is given by
    \begin{equation*}
        \widehat{i\cdot \cL}=\widecheck{L}.
    \end{equation*}
\end{proposition}
\begin{proof}
    The proposition follows directly from the identities of equation \eqref{prop_cx_action} in Appendix \ref{apcxdirac}.
\end{proof}
A straightforward verification lead us to the following identification:
\begin{equation*}\label{lcheckid}
\widecheck{L}=\{ X+\xi\st \exists\eta\in T^*M \:\:\text{such that}\:\: X+\xi+i\eta\in L\}
\end{equation*}
\begin{proposition}
    If the rank of $\pr_{TM} L$ is constant, then $\widecheck{L}$ is a Dirac structure.
\end{proposition}
\begin{proof}
We regard $L$ as a complex Lie algebroid. By the hypothesis and by Proposition \ref{lreisalgebroid}, the subbundle $L^{\rea}=L\cap (TM\oplus T^*_{\C}M)$ is a real Lie algebroid.
    By the identification in \eqref{lcheckid}, we have \begin{equation*}
        \widecheck{L}=\pr_{\T M}L^{\rea}.
    \end{equation*}
    Consequently, $\widecheck{L}$ is smooth.
    
    For the involutivity, let $X+\xi, Y+\eta\in \Gamma(\widecheck{L})$. Then, there exists $\xi_{0},\eta_{0}\in \Gamma(T^{*}M)$ such that 
    \begin{equation*}
        X+\xi+i\xi_{0},\;Y+\eta+i\eta_{0}\in \Gamma(L).
    \end{equation*}
Since $L$ is involutive 
\begin{align*}
[X+\xi+i\xi_{0},Y+\eta+i\eta_{0}]=[X,Y]+(\lie_{X}\eta-\iota_{Y}d\eta)+i(\lie_{X}\eta_{0}-\iota_{Y}d\xi_{0})\in \Gamma(L).
\end{align*}
By the identification in \eqref{lcheckid}, $[X+\xi,Y+\eta]\in \Gamma(\widecheck{L}).$
\end{proof}
For a complex Poisson bivector $\pi=\pi_1+i\pi_2$, we have:
\begin{equation*}
\widecheck{L_{\pi}}=\{\pi_1(\xi_1)-\pi_2(\xi_2)+\xi_1\st \pi_2(\xi_1)+\pi_1(\xi_2)=0\}.
\end{equation*}
As a consequence, we obtain Dirac-geometric interpretation of the space $\pi_2(\ker\pi_1)$:
\begin{corollary}
Let $\pi=\pi_1+i\pi_2$ be a complex Poisson bivector. Then, \begin{equation*}
\ker\widecheck{L_{\pi}}= \pi_2(\ker\pi_1).
\end{equation*}
\end{corollary}
\begin{proof}
    Straightforward.
\end{proof}
\subsection{Complex tangent sums}
\begin{definition}
Given two lagrangian families $L_1$ and $L_2$, we define their {\em complex tangent sum} by:
 \begin{equation*}
        L_1\star_{\C}L_2=(L_1)_{\C}\star(i\cdot(L_2)_{\C}).
    \end{equation*}
    
\end{definition}
    Observe that the complex tangent sum is always a quasi-real lagrangian family. Furthermore, it is easy to see that 
    \begin{align*}
    \widecheck{L_1\star_{\C}L_2}&=L_1\\
        \widehat{L_1\star_{\C}L_2}&=L_2.
    \end{align*}
    \begin{example}
        Let $\omega_1, \omega_2\in \Omega^2(M)$. Then, 
        \begin{equation*}
    L_{\omega_1}\star_{\C}L_{\omega_2}=L_{\omega_1+i\omega_2}.
    \end{equation*}
    \end{example}
\begin{example}
    Let $L_1=L(D_1,\omega_1)$ and $L_2=L(D_2,\omega_2)$ be two regular Dirac structures. Then, 
     \begin{equation*}
        L_1\star_{\C}L_2=L((D_1\cap D_2)_{\C}, \omega_1+i\omega_2).
    \end{equation*}
    Note that, if $D_1\cap D_2$ is not smooth, then neither $L_1\star_{\C}L_2$. 
\end{example}
 As illustrated in the previous example, the complex tangent sum is not necessarily smooth, even when both factors are complex Dirac structures.
\begin{example}
 Using the complex tangent sum, we can regard a Poisson-Nijnhuis as a complex Dirac structure in a new way:
 Let $(\sigma, N)$ be a Poisson-Nijenhuis. Then, consider
\begin{equation*}
    L_{(\sigma,N)}=L_{\sigma}\star_{\C} L_{\sigma_N},
\end{equation*}
the {\em quasi-real lagrangian family associated to a Poisson-Nijenhuis structure}.
\end{example}
Recall that a Poisson-Nijenhuis structure $(\sigma,N)$ defines a complex Poisson bivector given by $\pi(\sigma,N)=\sigma+i\sigma_N$. With this, we have the following:
\begin{proposition}\label{check_pn}
    Let $(\sigma, N)$ be a Poisson-Nijenhuis structure. Then,
    \begin{equation*}
        \widecheck{L_{\pi(\sigma,N)}}=e^{\sigma_{N^2}}L_{\sigma},
    \end{equation*}
    where $\sigma_{N^2}=N^2\circ\sigma$ and the map $e^{\sigma_{N^2}}$ is a $\beta$-transformation.
\end{proposition}
\begin{proof}
    The proof relies in the trick $(N^2+Id)\sigma(\xi_1)=\sigma(\xi_1)-\sigma_N(\xi_2)$ of Proposition \ref{pn_qr}.
\end{proof}
Since $\widecheck{ L_{(\sigma,N)}}=L_{\sigma}$, the Dirac structure $L_{\pi(\sigma,N)}$ is different to $L_{(\sigma,N)}$. Moreover, note that the leaves of the complex presymplectic foliation correlated to $L_{(\sigma,N)}$ coincide with the leaves of the symplectic foliation associated to $\sigma$. Each leaf $S$ of this foliation is equipped with the complex two-form $\omega_S+i(\omega_S)_N$, which arises from the symplectic-Nijenhuis structure $(\omega, N|_{TS})$ on the leaf $S$.
\begin{example}
    More generally, given a pair of bivectors $(\pi_1,\pi_2)$, we define the lagrangian family
    \begin{equation*}
        L_{\pi_1,\pi_2}=L_{\pi_1}\star_{\C}L_{\pi_2}.
    \end{equation*}
    Note that the family $L_{\pi_1,\pi_2}$ is quite different from the graph $L_{\pi_1+i\pi_2}$, as we observed in the case of Poisson-Nijenhuis structures.
\end{example}
\begin{proposition}
    Let $L_1, L_2$ be two real lagrangian families. Then $L_1\star L_2$ is smooth if and only if $L_1\star_{\C}L_2$ is smooth.
\end{proposition}
\begin{proof}
  Let $p\in M$. Consider 
  \begin{equation*}
      e=v_1+iv_2+\xi_1+i\xi_2+\eta_1+i\eta_2\in L_1\star_{\C}L_2|_p.
  \end{equation*}
   Then, 
   \begin{equation*}
       v_1+iv_2+\xi_1+i\xi_2\in (L_1|_p)_{\C}\text{ and }v_1+iv_2-i(\eta_1+i\eta_2)\in (L_2|_p)_{\C}.
   \end{equation*}
   Consequently, 
   \begin{equation*}
   v_1+\xi_1, v_2+\xi_2\in L_1|_p\text{ and } v_1-\eta_2, v_2-\eta_1\in L_2|_p.
   \end{equation*}
   Now consider  $e_1=v_1+\xi_1-\eta_2,\: e_2= v_2+\xi_2-\eta_1$. Note that $e\in L_1\star_{\C}L_2$ if and only if $e_1, e_2\in L_1\star L_2$. 
   
   If we assume that $L_1\star L_2$ is smooth, then there exist smooth extensions for $e_1$ and $e_2$ that are tangent to $L_1\star L_2$, i.e., local extensions $\widetilde{v}_1,\widetilde{v}_2 \in \X(M)$ and $\widetilde{\xi}_1,\widetilde{\xi}_2, \widetilde{\eta}_1, \widetilde{\eta}_2\in \Omega^1(M)$ of $v_1,v_2,\xi_1,\xi_2, \eta_1, \eta_2$ such that 
   \begin{equation*}
\widetilde{v}_1+\widetilde{\xi}_1,\: \widetilde{v}_2+\widetilde{\xi}_2\in \Gamma(L_1)\text{ and } \widetilde{v}_1+\widetilde{\eta}_2,\: \widetilde{v}_2-\widetilde{\eta}_1\in \Gamma(L_2).
\end{equation*}
Therefore, 
\begin{equation*}
\widetilde{e}=\widetilde{v}_1+i\widetilde{v}_2+\widetilde{\xi}_1+i\widetilde{\xi}_2+\widetilde{\eta}_1+i\widetilde{\eta}_2\in \Gamma(L_1\star_{\C}L_2).
\end{equation*}
is a local extension of $e$ that is tangent to $L_1\star_{\C}L_2$.
 
   Now assume that $L_1\star_{\C}L_2$ is smooth. Let $p\in M$, and
    \begin{equation*}
    f=v+\xi+\eta\in L_1\star L_2|_p 
   \end{equation*}
   such that $v+\xi \in L_1|_p$ and $v+\eta\in L_2|_p$. Then, 
   \begin{equation*}
       e=v+iv+\xi+i\xi-\eta-i\eta\in L_1\star_{\C}L_2|_p
   \end{equation*}
   and it admits a local extension 
   \begin{equation*}  \widetilde{e}=\widetilde{v}+i\widetilde{v'}+\widetilde{\xi}+i\widetilde{\xi'}-\widetilde{\eta}-i\widetilde{\eta'}
   \end{equation*}
   that is tangent to $ L_1\star_{\C}L_2$. Consequently, the section $\widetilde{f}=\widetilde{v}+\widetilde{\xi}+\widetilde{\eta}$ is a local extension of $f$ that is tangent to $L_1\star L_2$.
\end{proof}
\begin{proposition}
    Every quasi-real lagrangian family can be expressed as the complex tangent sum of two (real) lagrangian families.
\end{proposition}
\begin{proof}
    Straightforward.
\end{proof}
\subsection{Associated quasi-real lagrangian family}
\begin{definition}
    Let $L$ be a complex Dirac structure. Its {\em associated quasi-real lagrangian family} $\widetilde{L}$ is defined by the formula
    \begin{equation*}
        \widetilde{L}=\widecheck{L} \star_{\C}\widehat{L}.
    \end{equation*}
\end{definition}
If $L=L(E, \varepsilon_L)$, where $E\subseteq\TCM$ and $\varepsilon_L: E\times E\to \C$ is skew-symmetric, then we have that 
\begin{equation}\label{explicit_hat}
    \widehat{L}=L(\Delta_L, {\omega_L}_{|_{\Delta_L}}),
\end{equation}
\begin{equation}\label{explicit_check}
    \widecheck{L}=L(\Delta_L, {B_L}_{|_{\Delta_L}})
\end{equation}
and thus
\begin{equation*}
    \widetilde{L}=L((\Delta_L)_{\C}, (\omega_L+i{B_L)}_{|_{\Delta_L}})=L((\Delta_L)_{\C},{\varepsilon_L}_{|_{(\Delta_L)_{\C}}}),
\end{equation*}

where $\varepsilon_L|_{\Delta_L}=B_L+i\omega_L$. From the previous identities, it is immediately clear that $\widetilde{L}$ is quasi-real.
\begin{example}
    Let $E\subseteq \TCM$ be a regular complex distribution and $L=L(E,0)$ its graph. Assume that $\Delta=E\cap TM$ is a regular distribution. Then, 
    \begin{equation*}
        \widehat{L}=\widecheck{L}=L(\Delta,0)\:\text{ and }\: \widetilde{L}=L(\Delta_{\C},0)
    \end{equation*}
    are lagrangian subbundles.
\end{example}

\begin{example}
    Let $\omega=\omega_1+i\omega_2\in \Omega_{\C}(M)$ be a complex two-form and let $L_\omega$ be its graph. Then
    \begin{equation*}
        \widecheck{L_\omega}=L_{\omega_1},\:\:\widehat{L_\omega}=L_{\omega_2}\:\:\text{ and }\:\: \widetilde{L_\omega}=L_\omega.
    \end{equation*}
\end{example}
\begin{proposition}
    A lagrangian family $L$ is quasi-real if and only if $\widetilde{L}=L$.
\end{proposition}
\begin{proof}
    Straightforward.
\end{proof}
\begin{remark}
    For a complex Poisson bivector $\pi=\pi_1+i\pi_2$, there are associated two complex Dirac structures: $L_{\pi}$ and $L_{\pi_1,\pi_2}$, the latter being a quasi-real lagrangian family. So it is natural to ask if $\widetilde{L_{\pi}}=L_{\pi_1,\pi_2}$.  We know that $\widetilde{L_{\pi}}$ is quasi-real and if $\widetilde{L_{\pi}}=L_{\pi_1,\pi_2}$, it would follows that $\widecheck{L_{\pi}}=L_{\pi_1}$. However, this is not true in general, for instance, consider a Poisson-Nijenhuis structure. By Proposition \ref{check_pn}, we have $\widecheck{L_{\pi(\sigma,N)}}=e^{\sigma_{N^2}}L_{\sigma}$. Consequently, the equality $\widecheck{L_{\pi(\sigma,N)}}=L_{\sigma}$ holds if $\ima\sigma^{\shrp}\subseteq \ker N^2$.
\end{remark}

\bigskip
The main reason we introduce all the preceding operations on Dirac structures, in particular the operation $\:\:\widetilde{\cdot}\:\:$, is to provide a more concise description of the complex presymplectic foliation associated to a complex Poisson bivector.
\begin{theorem}\label{cx_presym_fol}
Let $\pi$ be a complex Poisson structure. Then,
\begin{equation}\label{cx_presymp_fol_as_dirac}
    \widetilde{L_{\pi}}=L((\deltpi)_\C,\Omega|_{(\deltpi)_{\C}}),
\end{equation}
that is, all the information of the complex presymplectic foliation is encoded in $\widetilde{L_{\pi}}$. Moreover, if $\pi$ is strongly regular, then $\widetilde{L_{\pi}}$ is a complex Dirac structure.
\end{theorem}
\begin{proof}
    The identities \eqref{explicit_hat} and \eqref{explicit_check} imply that $ \widetilde{L_{\pi}}=L((\deltpi)_\C,\Omega|_{(\deltpi)_{\C}})$. Assume now that $\pi$ is strongly regular. The smoothness of $\widetilde{L_{\pi}}$ follows from the regularity of $\deltpi$. Since $\pi$ has also constant order, $i\cdot(\widehat{L_{\pi}})_{\C}$ and $(\widecheck{L_{\pi}})_{\C}$ are complex Dirac structures. Consequently, their tangent product, namely $\widetilde{L_{\pi}}$, is involutive.
\end{proof}
\begin{corollary}\label{cx_symp_fol_qr}
    If $\pi$ is a quasi-real Poisson bivector, then the following identity holds
    \begin{equation*}
        \widetilde{L_{\pi}}=L((\Dpi)_\C,\Omega)
    \end{equation*}
    as lagrangian families. Moreover, if $\pi$ has constant order $\widetilde{L_{\pi}}$ is smooth.
\end{corollary}
What both Theorem \ref{cx_presym_fol} and Corollary \ref{cx_symp_fol_qr} are concluding is that the complex presymplectic two-forms defined on the leaves of the underlying complex presymplectic foliation glue smoothly along the leaves. 
\begin{remark}
    Further properties of $\widehat{L}$, $\widecheck{L}$ and $\widetilde{L}$, and the complex tangent sums will be explore in future work.
\end{remark}
     Given two lagrangian families $L_1$ and $L_2$, {\em the cotangent product} of $L_1$ with $L_2$, see \cite{frejlich2024dirac}, is defined as
    \begin{equation*}
    L_1 \circledast L_2=\{X_1+X_2+\eta\st X_1+\eta\in L_1\:\:\text{and}\:\: X_2+\eta\in L_2 \}.
\end{equation*}
Using the cotangent product of lagrangian families, we can obtain many lagrangian families associated to a complex Dirac structures.
\begin{definition}
Given two lagrangian families $L_1$ and $L_2$, the {\em complex cotangent sum} of $L_1$ with $L_2$ is defined as:
\begin{equation*}
    L_1\circledast_{\C}L_2=(L_1)_{\C}\circledast (i\cdot(L_2)_{\C}).
\end{equation*}
\end{definition}
Therefore, by replacing the tangent product with the cotangent product in equations \eqref{lhat} and \eqref{lcheck}, we obtain two additional real lagrangian families:
\begin{align*}
    \widehat{L}^{cot}&=\frac{1}{2i}\cdot(L\circledast (-1)\cdot\overline{L})\cap \T M\\
    \widecheck{L}^{cot}&=\frac{1}{2}\cdot(L\circledast \overline{L})\cap \T M.
\end{align*}
And a complex lagrangian family:
 \begin{equation*}
        \widetilde{L}^{cot}=\widecheck{L} \circledast_{\C}\widehat{L}.
    \end{equation*}
Observe that $\widetilde{L}^{cot}$ is not in general a quasi-real lagrangian family.
\begin{remark}
    The complex tangent and cotangent sums extend the operations 
    \begin{align*}
    (\omega_1,\omega_2)\in \Omega^2(M)\times \Omega^2(M)&\mapsto\omega_1+i\omega_2\in \Gamma(\wedge^2 \CCM))\\
    (\pi_1,\pi_2)\in \X^2(M)\times \X^2(M)&\mapsto\pi_1+i\pi_2\in \Gamma(\wedge^2 \CCM))
    \end{align*} 
    to Dirac structures, respectively. These constructions are exploiting the dual nature of Dirac structures, i.e, Dirac structures behave as two-forms (for example, under backward images) and like bivectors (for example, under forward images). As mentioned in \cite{frejlich2024dirac}, the cotangent product does not provide new information for Poisson structures but it does for presymplectic forms. Instead, for Poisson bivectors the most suitable operation is the tangent product.
\end{remark}

\section{Complex Poisson structures with constant real index}
\subsection{The real index of a complex bivector}
\begin{definition}
    Let $\pi$ be a complex bivector. The {\em real index} of $\pi$ is given by the $\mathbb{N}$-valued function
\begin{equation*}
    p\in M\mapsto \text{real-index}_p (\pi)=\rk K_{\pi}, 
\end{equation*}
where
    \begin{equation*}
    K_{\pi}=L_{\pi}\cap \T M
\end{equation*}
is the {\em real part of $L_{\pi}$}.
\end{definition}

Note that 
\begin{align*}
     K_{\pi}&=\{\pi(\xi+i\eta)+\xi+i\eta\}\cap \T M\\
     &=\{\pi_1(\xi)+\xi\st \pi_2(\xi)=0 \}\subseteq \gr(\pi_1)
\end{align*}
   and that $K_{\pi}$ is isotropic in $\T M$ with respect to the canonical pairing.
Thus, we have:
\begin{lemma}\label{realpart}
    The real part of $L_{\pi}$ is
    \begin{equation*}
        K_{\pi}=\gr(\pi_1|_{\ker\pi_2})=e^{\pi_1}\ker\pi_2.
    \end{equation*}
    Moreover, the real index of $L_{\pi}$ is $\rk\ker\pi_2$.
\end{lemma}
\begin{corollary}
    A complex bivector $\pi=\pi_1+i\pi_2$ has constant real index if and only if $\ker \pi_2$ is a vector subbundle of $T^*M$.
\end{corollary}
Let $\widehat{L_{\pi}}$ be the lagrangian family associated to $L_{\pi}$ and let $\omega_{\Delta}$ be the point-wise two-form regarded in the whole distribution $\deltpi$. By Lemma \ref{proj_K} in Appendix \ref{apcxdirac}, we have that $\pr_{TM} K_{\pi}=\ker \omega_{\Delta}$. So we have the following:
\begin{corollary} The following equalities holds:
    \begin{equation*}
    \ker \widehat{L_{\pi}}=\ker\omega_{\Delta}=\pi_1(\ker\pi_2).
    \end{equation*}
\end{corollary}
Observe that until now, we have not assumed the smoothness of $\widehat{L_{\pi}}$ nor the integrability of~$\pi$.

\medskip
Let $\pi=\pi_1+i\pi_2$ be a complex Poisson bivector such that $\ker\pi_2$ is a subbundle, i.e., $\pi$ has constant real index. Since $L_{\pi}$ is involutive, $K_{\pi}$ is involutive with respect to the Courant-Dorfmann bracket. Furthermore, $K_{\pi}$ is an isotropic and involutive subbundle of $\T M$, implying that the triple $(K_{\pi}, [\cdot,\cdot]|_{K_{\pi}}, \pr_{TM})$ defines a Lie algebroid, where $[\cdot,\cdot]$ is the Courant-Dorfmann bracket of $\T M$. Hence, by the previous corollary $\pi_1(\ker\pi_2)$ is the image of the anchor map of $K_{\pi}$ and thus it is integrable.
\begin{corollary}
 If $\ker\pi_2$ is a subbundle of $T^*M$, then $\pi_1(\ker\pi_2)$ is integrable.  
\end{corollary}

Using the real index we observe that $\pi$ and $i\pi$ are quite different. For example, 
\begin{align*}
    \text{real-index }(\pi)&=\rk\ker\pi_2\text{ and}\\
    \text{real-index }(i\pi)&=\rk\ker\pi_1
\end{align*}
Hence, $L_\pi$ can have constant real index while $L_{i\pi}$ does not. The reason is that
\begin{equation*}
    L_{i\pi}=i\bullet L_{\pi}.
\end{equation*}
Here $\bullet$ is the cotangent version of the action of the scalar on the space of lagrangian:
\begin{equation*}
    z\bullet L=\{zX+\xi\st X+\xi\in L\},
\end{equation*}
where $z\in\C$ and $L$ is a lagrangian family. When $z\neq 0$, $z\bullet L=1/z\cdot L$. We observe that the multiplication by $i$, changes all the real invariants of $L_{\pi}$, the real index included.

\subsection{Reduction}
A natural question raised in the previous question is that if we can somehow annihilate the real index of complex Poisson bivector to obtain a complex Poisson bivector with zero real index, i.e., a generalized complex structure. To address this question, we develop a more general reduction scheme for complex Poisson bivectors.

Given a real Dirac structure $L$, one way to obtain a Poisson structure from $L$ is by taking a ``quotient'' of the manifold by the kernel of $L$. However, in the complex setting this generally does not work. 

The {\em kernel} of a complex Dirac structure $L$ is 
\begin{equation*}
\ker L=L\cap \TCM\subseteq \TCM,
\end{equation*}
note that $\ker L$ is always involutive. 
Let $\varphi:M\to N$ be a map and $L_M$ a complex Dirac on $M$. Recall that in general the forward image of $L_M$, $\varphi_! L_M$ is a lagrangian family inside the bundle $\varphi^* \T_{\C} M$ over $M$ and it does not necessarily define a Dirac structure over $N$.

\medskip
Let $\varphi:M\to N$ be a submersion, note that $\ker \varphi_{!}L=\varphi_*(\ker L)$. Since $\ker L$ is an involutive (possibly singular) complex distribution, we have two options for the reduction: 
\begin{enumerate}
    \item Assume that $\Delta({\ker L})$ is regular: in this case $\Delta({\ker L})$ defines a foliation $\F_{\Delta}$. Suppose that $\F_{\Delta}$ is a simple foliation with leaf space $N_{\Delta}$ and with submersion $\varphi_{\Delta}:M\to N_{\Delta}=M/\F_{\Delta}$ whose fibers are the leaves of $\F_{\Delta}$. Note that $\ker ((\varphi_{\Delta})_! L)\neq 0$. Actually we have $(\varphi_{\Delta})_! L\cap TN_{\Delta}=0$, so $(\varphi_{\Delta})_! L$ is not necessarily a complex Poisson structure. 

\item Assume that $D(\ker L)=\pr_{TM} \ker$ is involutive and regular. Then, we have a foliation $\F_D$ associated to $D(\ker L)$. Now assume that $\F_D$ is a simple foliation with leaf space $N_D$ and submersion $\varphi_D:M\to N_D=M/\F_D$. In this case, we have $\ker (\varphi_D)_! L=0$ and thus $(\varphi_D)_! L$ defines a complex Poisson structure.
\end{enumerate}
In summary, we have the following:
\begin{proposition}\label{d_integrablereduction}
 Let $L$ be a complex Dirac structure such that $D(\ker L)$ is an involutive regular distribution. If the foliation associated to $D(\ker L)$ is simple with submersion $\varphi: M\to N_D$, then $\varphi_! L$ is the graph of a complex Poisson bivector.
\end{proposition}
Now we present a complex parallel of Proposition 9.4 of \cite{meinrenken2017introduction}.
\begin{proposition}\label{descending_poisson}
Let $M$ be a manifold with complex Poisson bivector $\pi$ and let $\varphi:M\to N$ be a submersion. Then $\pi$ descends to a complex Poisson bivector $\pi_N$, making $\varphi$ a Poisson map if and only if $\Ann (\ker T_{\C}\varphi)$ is a complex Lie subalgebroid of $T^*_{\C}M$.
\end{proposition}
\begin{proof}
Assume that $\Ann\ker T_{\C}\varphi$ is a Lie subalgebroid. Take any $f,g\in C^{\infty}(N)$ (real functions!), then
\begin{equation*}
[T_{\C}\varphi^*f, T_{\C}\varphi^*g]=T_{\C}\{\varphi^*f, \varphi^*g\}=T_{\C}(\{\varphi^*f, \varphi^*g\}_1+i\{\varphi^*f, \varphi^*g\}_2)\in \Ann\ker T_{\C}\varphi.
\end{equation*} 
Hence, $T_{\C}(\{\varphi^*f, \varphi^*g\}_1+i\{\varphi^*f, \varphi^*g\}_2)$ vanishes on $\ker T\varphi$ and consequently $T\{\varphi^*f, \varphi^*g\}_1$ and $T\{\varphi^*f, \varphi^*g\}_2$ vanish on $\ker T\varphi$. Thus, $\{\varphi^*f, \varphi^*g\}_1$ and $\{\varphi^*f, \varphi^*g\}_1$ are constant along the fibers of $\varphi$ and so both are pullback of functions on $N$, let us call these functions by $\{f,g\}^1_N$ and $\{f,g\}^2_N$. Therefore, we define the bracket 
\begin{equation*}
\{f,g\}_N=\{f,g\}^1_N+i\{f,g\}^2_N
\end{equation*} 
and then extends it to $C^{\infty}(N,\C)$ by $\C$-linearity. This bracket satisfy that $\{\varphi^*f, \varphi^*g\}=\varphi^*\{f,g\}_N$, and that makes $\{\cdot,\cdot\}_N$ a complex Poisson structure on $N$.
\end{proof}

Let $\pi=\pi_1+i\pi_2$ be a complex Poisson such that $\ker\pi_2$ is regular. Then, the distribution $C=\ima(\pi_2)\subseteq TM$ satisfies that $\Ann C=\ker\pi_2$. In general, the image of a bivector is not integrable; when that is the case, we say that such a bivector is {\em involutive}. However, there exist involutive bivectors that are not necessarily Poisson, for example twisted Poisson structures \cite{vsevera2001poisson}, quasi-Poisson structures \cite{alekseev2002quasi}, etc. In general, an involutive bivector defines a possibly singular foliation where each leaf inherits a non-degenerate, non necessarily closed two-form.

\begin{example}
    Consider the complex Poisson bivector from Examples \ref{ex_non_biham} 
    \begin{equation*}
        \pi=\px\wedge\py+y\py\wedge\pz+i(\px\wedge\pz+z\py\wedge\pz).
    \end{equation*}
    The bivector $\pi$ is a strongly regular complex Poisson structure. However,
 a straightforward computation shows that
    \begin{align*}
        %\ima\pi_1^{\shrp}&=span\la \py, -\px+y\pz, -y\py\ra,\\
        \ima\pi_1^{\shrp}&=span\la \py, -\px+y\pz \ra,\\
        \ima\pi_2^{\shrp}&=span\la \pz, \px+z\py\ra
    \end{align*}
    are not involutive, thus $\pi_1$ and $\pi_2$ are not involutive bivectors.
Therefore, we observe that even with the strongest regularity condition, the involutivity of the factors $\pi_1$ or $\pi_2$ is not ensured.

\end{example}

An application of Proposition \ref{descending_poisson} to involutive bivectors is the following:
\begin{corollary}\label{reductionbykernel2}
Let $\pi=\pi_1+i\pi_2$ be a complex Poisson such that $\pi_2$ is involutive and $\ker\pi_2$ is regular. If $C=\ima(\pi_2)$ defines a simple foliation with leaf space $N$ then  $\pi$ descend to a complex Poisson bivector $\pi_N=\pi_N^1+i\pi_N^2$ satisfying that $\pi_N^2$ is non-degenerate.
\end{corollary}
\begin{proof}
    Let $\varphi: M\to N=M/\thicksim_{C}$ be the quotient map associated to the simple foliation of $C$. Observe that $\Ann T_{\C}\varphi=(\ker\pi_2)_{\C}$. By Proposition \ref{kernelrealbivector}, $\Ann T_{\C}\varphi=(\ker\pi_2)_{\C}$ is a Lie subalgebroid of $(\CCM)_{\pi}$. Thus, by Proposition \ref{descending_poisson}, the bivector $\pi$ descends to a complex Poisson bivector $\pi_N$ on $N$. The map $\varphi$ is now a Poisson morphism from $\pi$ to $\pi_N$, so we have that $\varphi_*\pi_2=\pi_N^2$ and therefore it follows that $\pi_N^2$ is non-degenerate.
\end{proof}
The previous results can be interpreted in a broader reduction scheme for complex Dirac structures, if $L$ is a complex Dirac structure such that $K_L=\rea L$ is a vector subbundle, we cannot reduce directly through $K_L$ but instead on $\pr_{TM}K_L$. 
\subsection{Complex Poisson bivectors with zero real index}
 After Lemma \ref{realpart}, we have that the real index of a complex Poisson structure $\pi=\pi_1+i\pi_2$ is $\rk\ker\pi_2$. In particular, a complex Poisson structure has real index zero if and only if $\ker\pi_2=0$, that is, $\pi_2$ is invertible. 

If $\pi$ has constant real index and $\pi_2$ is an involutive bivector, then, by Corollary \ref{reductionbykernel2}, we can reduce $\pi$ locally to a complex Poisson bivector $\widetilde{\pi}$ such that $L_{\widetilde{\pi}}$ is a generalized complex structure. 

\medskip
Now consider a complex Poisson structure $\pi=\pi_1+i\pi_2$ with real index zero
\begin{equation*}
    L_{\pi}=\{\pi_1\xi_1-\pi_2\xi_2+\xi_1+i(\pi_1\xi_2+\pi_2\xi_1+\xi_2)\st \xi_1+i\xi_2\in \CCM \}.
\end{equation*}
Since $\pi_2$ is invertible, given any $X\in TM$ we consider $\xi_2=\pi_2^{-1}(\pi_1\xi_1-X)$. 
So we can rewrite $L_{\pi}$ as 
\begin{equation*}
     L_{\pi}=\{X+\xi+i(\pi_1\pi_2^{-1}(\pi_1\xi-X)+\pi_2\xi+\pi_2^{-1}(\pi_1\xi-X))\st X+\xi\in \T M \}.
\end{equation*}
Consequently, we have obtained the following:
\begin{theorem}\label{cxp_gcs}
Let $\pi=\pi_1+i\pi_2$ be a complex Poisson bivector. Then $L_{\pi}$ is a generalized complex structure if and only if $\pi_2$ is invertible. In such case, the generalized complex map $\J_{\pi}: \T M\to \T M$ associated to $L_{\pi}$ is given by:
\begin{equation*}
   \J_{\pi}=\begin{pmatrix}
       -\pi_1\pi_2^{-1} & \pi_1\pi_2^{-1}\pi_1+\pi_2\\
       -\pi_2^{-1} & \pi_2^{-1}\pi_1
   \end{pmatrix} .
\end{equation*}
In particular, the bivector $\sigma=\pi_1\pi_2^{-1}\pi_1+\pi_2$ is the Poisson bivector associated to $\J_{\pi}$.
\end{theorem}
The previous example provides new examples of generalized complex structures with jumping type, since $\pi_1$ is not necessarily regular.

\medskip
Recall that a bi-Hamiltonian structure $(\pi_1,\pi_2)$ such that $\pi_2$ is symplectic is equivalent to a Poisson-Nijenhuis structure $(\pi_2,N)$ with Nijenhuis map given by $N=\pi^{\shrp}_1\circ(\pi^{\shrp}_2)^{-1}:TM\to TM$. Note that $\pi_1$ is not necessarily invertible and so $(\pi_1,\pi_2)$ does not necessarily define a symplectic-Nijenhuis.
\begin{corollary}
    Let $(\sigma, N)$ be a Poisson-Nijenhuis structure such that $\sigma$ is invertible. Then, the bivector $\pi=i\pi(\sigma,N)^c=\sigma_N+i\sigma$ defines a generalized complex structure with associated map 
    \begin{equation*}
   \J_{\sigma,N}=\begin{pmatrix}
       -N & \sigma_{N^2}+\sigma\\
       -\sigma^{-1} & N^*
   \end{pmatrix} .
\end{equation*}
\end{corollary}

\section{Local form}
\subsection{Local models}
Local models for Poisson manifolds were introduced in \cite{frejlich2017normal}. In this section, we introduce local models for complex Poisson structures.
\begin{definition}
    A {\em complex symplectic vector bundle} is a complex vector bundle $E\to M$ equipped with a smooth family of non-degenerated complex bilinear forms on each fiber. 
\end{definition}
\begin{definition}
   Let $(E,\sigma)$ be a complex symplectic vector bundle. An {\em extension of $\sigma$} is a complex-valued two-form $\widetilde{\sigma}\in \Omega^2(E,{\C})$ such that $\widetilde{\sigma}|_{TE|_M}=\sigma$.
\end{definition}
Note that any extension $\widetilde{\sigma}$ extends to a complex two-form $\widetilde{\sigma}^{\C}$ by complex bilinearity.
In this case, we have that 
\begin{equation*}
    \widetilde{\sigma}^{\C}|_{T_{\C}E|_M}=\sigma_{\C}.
\end{equation*}
 Any complex symplectic vector bundle $E\to M$ is of the form $E=F\times_{Sp(2n,\C)}\C^{2n}$, where $F$ is the complex symplectic frame bundle and $\C^{2n}$ is equipped with the canonical complex symplectic structure. The existence of extensions is ensured by applying the minimal coupling construction of Sternberg \cite{sternberg1977minimal} and Weinstein \cite{weinstein1978universal} to the associated bundle $E=F\times_{Sp(2n,\C)}\C^{2n}$ (see, for example, \cite{meinrenken2017introduction}). 

\medskip
Now, consider a complex symplectic vector bundle $p:E\to M$ with a complex symplectic structure $\sigma$. Assume that $\pi_M$ is a complex Poisson structure over $M$. With this data, we construct the following complex Dirac structure on $E$:
\begin{equation*}
    L(\widetilde{\sigma})=e^{\widetilde{\sigma}^{\C}}p^!\gr(\pi_M)\subseteq \T_{\C}E.
\end{equation*}
\begin{proposition}\label{local_model}
 Let $(E,\sigma)$ be a complex symplectic vector bundle over $M$ and let $\pi_M$ be a complex Poisson structure on $M$. Then there exists a neighborhood $U(\widetilde{\sigma})$ of $M$ in $E$ and a complex Poisson bivector $\pi(\widetilde{\sigma})$ on $U(\widetilde{\sigma})$ such that
 \begin{equation*}
     L(\widetilde{\sigma})|_{U(\widetilde{\sigma})}\cong \gr\pi(\widetilde{\sigma})
 \end{equation*}
 and \begin{equation}
     \pi(\widetilde{\sigma})|_M=\pi_M+(\sigma_{\C})^{-1}.
 \end{equation}
\end{proposition}
\begin{proof}
   The proof is very similar to the real case; for completeness, we include it.   
    Note that 
    \begin{align*}
        p^!\gr(\pi_M)|_M&=\{X+p^*\xi \st Tp(X)=\pi_M(\xi)\}|_M\\
        &=\{\pi_M(\xi)+Y_1+iY_2+\xi\st Y_1, Y_2\in E\text{ and } \xi \in T^*_{\C}M\},
    \end{align*}
    where the identity of the second line is given by the fact that $\ker\T_{\C}p|_M=E_{\C}$ and $p^*\xi|_M=\xi$.

    \begin{align*}
        L(\widetilde{\sigma})|_M&=\{\pi_M(\xi)+Y_1+iY_2+\xi+\iota_{Y_1+iY_2}\sigma_{\C}\st Y_1, Y_2\in E\text{ and } \xi \in T^*_{\C}M \}\\
        &=\{\pi_M(\xi)+(\sigma_{\C})^{-1}(\eta)+\xi+\eta\st \eta\in E^*_{\C}\text{ and } \xi \in T^*_{\C}M \}\\
        &=\{(\pi_M+(\sigma_{\C})^{-1})(\zeta)+\zeta\st \zeta\in T^*_{\C}E|_M\}.
    \end{align*}
 Note that $ L(\widetilde{\sigma})|_M\cap T_{\C}E|_M=0$ and then by Corollary \ref{open_condition}, there exists a neighborhood $U(\widetilde{\sigma})$ of $M$ and a complex Poisson structure $\pi(\widetilde{\sigma})$ on $U(\widetilde{\sigma})$ such that $L(\widetilde{\sigma})|_{U(\widetilde{\sigma})}=\gr(\pi(\widetilde{\sigma}))$. 

 Furthermore, it follows from the previous characterization of $L(\widetilde{\sigma})|_M$ that 
 \begin{equation*}
     \pi(\widetilde{\sigma})|_M=\pi_M+(\sigma_{\C})^{-1}.
 \end{equation*}
 \end{proof}
\subsection{Normal forms for complex Dirac structures}
The normal form for Dirac structures is well known. For complex Dirac structures, it depends on the existence of a special kind of section:
\begin{proposition}\label{normal_cx_dirac}\cite[Theorem 5.1 and Remark 5.2]{bursztyn2019splitting}
    Let $L$ be a complex Dirac structure and $N\xhookrightarrow{\iota}M$ a submanifold transversal to the anchor map. Assume that there exists a section $\epsilon=X+\xi_1+i\xi_2\in \Gamma(L)$ such that $X\in \X(M)$, $\xi_1,\xi_2\in \Omega^1(M)$, $\epsilon|_N=0$, $X$ is Euler-like along $N$ with tubular neighborhood $\psi:\nu_N\to U\subseteq M$. Then, $\T_{\C}\psi: \T_{\C}\nu_N\to \TTCM$ restricts to an isomorphism of complex Dirac structures
\begin{equation*}
    (p^!\iota^! L)^{\omega_1+i\omega_2}\to L|_U,
\end{equation*}
    where $\omega_1+i\omega_2\in \Omega^2_{\C}(\nu_N)$ is given by:
    \begin{equation*}
        \omega_1+i\omega_2=\int_{0}^{1}\frac{1}{\tau}\kappa^{*}_{\tau}(d\xi_1+id\xi_2)d\tau.
    \end{equation*}
\end{proposition}

In \cite{aguero2022complex}, assuming that the complex Dirac structure has constant order and real index, we can always construct sections such as in Proposition \ref{normal_cx_dirac}, but only locally. Here, we present an improvement.

\begin{corollary}\label{normal_cx_dirac_co}
 Let $L$ be complex Dirac structure of constant order and let $N\xhookrightarrow{\iota}M$ be a submanifold transversal to the anchor map. Then there always exists a section $\epsilon=X+\xi_1+i\xi_2\in \Gamma(L)$ such that $X\in \X(M)$ is Euler-like along $N$, $\xi_1,\xi_2\in \Omega^1(M)$. Furthermore, $\epsilon$ provides a normal form for $L$ as in the previous proposition.     
\end{corollary}
\begin{proof}
Since $L$ has constant order, by Proposition \ref{lreisalgebroid}, 
\begin{equation*}
    L^{\rea}=L\cap (TM\oplus T_{\C}^*M)
\end{equation*}
 is a Lie algebroid. By Lemma \ref{existence_esp_section}, there always exists a section $\epsilon=X+\xi_1+i\xi_2\in \Gamma(L)$ such that $X\in \X(M)$ ($X$ is real!) is Euler-like along $N$ and $\xi_1,\xi_2\in \Omega^1(M)$.
\end{proof}

\begin{corollary}
 Let $L$ be complex Dirac structure with constant order and $m\in M$. Consider a submanifold $N\xhookrightarrow{\iota}M$ containing $p$ and such that $T_m N\oplus P=T_m M$, here $P=\Delta_m$. Then, there exist a nighborhood of $m$ and a two-form $\omega\in \Omega_{\C}^2(N\times P)$ such that
\begin{equation*}
    L|_U\cong e^{\omega}(\iota^! L\times T_{\C}P).
\end{equation*}
\end{corollary}

\subsection{Normal form for a complex Poisson manifold}
\begin{definition}
Let $(M,\pi)$ be a complex Poisson manifold. A {\em complex cosymplectic transversal submanifold $N$} is a submanifold of $M$ satisfying 
    \begin{equation}\label{cx_cosymp_id}
    \pi(\Ann T_{\C}N)\oplus T_{\C}N=\TCM|_N.
\end{equation}
\end{definition}
\begin{proposition}\label{backw_inc_cx_cosymp}
    If $N$ is a complex cosymplectic transversal submanifold of $M$, then $\iota^! \gr (\pi)$ is the graph of a complex Poisson structure on $N$, that we denote by $\pi_N$.
\end{proposition}
\begin{proof}
   Straightforward.
\end{proof}
For a complex cosymplectic transversal submanifold $N$, the complex vector bundle 
\begin{equation*}
 \pi(\Ann T_{\C}N)\cong(\nu_N)_{\C}
 \end{equation*}
 inherits the structure of a complex symplectic vector bundle $((\nu_N)_{\C}, \widetilde{\Omega})$ defined as
 \begin{equation}\label{normal_bundle_cx_symp}
     \widetilde{\Omega}(\pi(\xi_1),\pi(\xi_2))=\pi(\xi_1, \xi_2),
 \end{equation}
 for $\xi_1, \xi_2\in \Ann T_{\C}N$.
\begin{definition}
    A {\em mixed submanifold $N$} is a submanifold of $M$ satisfying the following conditions:
    \begin{equation}\label{pi_1_sub}
    \pi_1(\Ann TN)\oplus TN=TM|_N
    \end{equation}
    \begin{equation}\label{pi_2_sub}
    \pi_2(\Ann TN)=0
    \end{equation}
\end{definition}
\begin{example}
    Let $\sigma$ be a real Poisson bivector and let $N\xhookrightarrow{\iota}M$ be a cosymplectic transversal submanifold. It is easy to see that $N$ is a mixed submanifold of $\pi=\sigma_{\C}$. 
\end{example}
\begin{example}
    Let $\sigma$ be a real Poisson bivector and let $\pi=\sigma+i\sigma$ be the diagonal complexification. It is easy to see that there is no mixed submanifold for $\pi$. However, the normal form for $\pi$ only depends on $\sigma$.
\end{example}
For a mixed submanifold $N$, the normal bundle $\nu_N$ inherits a symplectic vector bundle structure from $\pi_1$, which we denote by $(\nu_N, \Omega)$.
It is clear that equations \eqref{pi_1_sub} and \eqref{pi_2_sub} imply that mixed submanifolds are complex cosymplectic transversals submanifolds. Thus, $(\nu_N)_{\C}$ inherits a complex symplectic bundle structure with the fiberwise two-form $\widetilde{\Omega}$ defined in equation \eqref{normal_bundle_cx_symp}. In this case, we have
\begin{equation*}
    ((\nu_N)_{\C}, \widetilde{\Omega})=((\nu_N)_{\C}, {\Omega}_{\C}).
\end{equation*}

For a mixed submanifold $N\xhookrightarrow{\iota}M$, Proposition \ref{backw_inc_cx_cosymp} implies that $\iota^! L_{\pi}=L_{\pi_N}$ for a complex Poisson bivector $\pi_N$ on $N$. Moreover, equation \eqref{pi_1_sub} implies that $\pi_1$ is tangent to $N$, so it restricts to a bivector $\pi_{1,N}\in \Gamma(\wedge^2 TN)$,  while equation \eqref{pi_2_sub}, implies that there exists a bivector $\pi_{2,N}\in \Gamma(\wedge^2 TN)$ satisfying that $\iota^!\gr(\pi_2)=\gr(\pi_{2,N})$. A straightforward computation shows the following:

\begin{proposition}
    Let $N\xhookrightarrow{\iota}(M,\pi)$ be a mixed submanifold. Then, 
    \begin{equation*}
        \pi_N=\pi_{1,N}+i\pi_{2,N}.
    \end{equation*}
\end{proposition}
\begin{proof}
    Straightforward.
\end{proof}
\begin{remark}
    Actually, given the flexibility of complex Poisson bivector we can consider many other type of mixed submanifolds. This will be developed in a future work.
\end{remark}
\begin{proposition}
    Let $N\xhookrightarrow{\iota}(M,\pi)$ be a mixes submanifold. Then, there always exists a section $\epsilon=X+\xi_1+i\xi_2\in \Gamma(\gr(\pi))$ such that $\epsilon|_N=0$ and $X$ is real and Euler-like along $N$.
\end{proposition}
\begin{proof}
Since $\pi_2(\Ann TN)=0$, we have $\Ann TN\subseteq A_{\pi}= \ker\rho_2$. Choose a complement $H$ for $\Ann TN$ in  $A_{\pi}$, so that
\begin{equation*}
    \Ann TN \oplus H=A_{\pi}.
\end{equation*}
Then, we have that 
\begin{equation*}
\xymatrix{
0\ar[r]&\iota^! A_{\pi}\ar[r]& A_{\pi}=\Ann TN \oplus H\ar[r]&\nu_N\ar[r]&0}.
\end{equation*}
Therefore, we can chose a section $s:\nu_N\to A_{\pi}|_N$ whose image lie in $\Ann TN$, that is $s: \nu_N\to \Ann TN$. Then, there exists a complex one form $\alpha=\alpha_1+i\alpha_2\in \Gamma(A_{\pi})\subseteq \Gamma(T^*_{\C}M)$ such that $d^N\alpha:\nu_N\to \Ann TN$ is precisely $s$. By Appendix \ref{apnormal}, the vector field $X=\pi^{\shrp}(\alpha)$ is Euler-like. 
\end{proof}
By Proposition \ref{normal_cx_dirac}, the section $\epsilon=X+\xi_1+i\xi_2\in \Gamma(\gr(\pi))$ provides a local model for $L=\gr(\pi)$.
Let $\psi:\nu_N\to U\subseteq M$ be the tubular neighborhood embedding associated to the Euler-like vector field $X$ and 
\begin{equation*}
B=\int_{0}^{1}\frac{1}{\tau}\kappa^{*}_{\tau}\psi^*(d\xi_1)d\tau, 
\end{equation*}
\begin{equation*}
\omega=\int_{0}^{1}\frac{1}{\tau}\kappa^{*}_{\tau}\psi^*(d\xi_2)d\tau,
\end{equation*}
where $B+i\omega\in\Omega^2_{\C}(\nu_N)$.
\begin{theorem}\label{normalform_cx_poisson}
    Let $\pi=\pi_1+i\pi_2\in \X(M)_{\C}$ be a complex Poisson structure with constant order and let $N\xhookrightarrow{\iota}M$ be a mixed submanifold. Choose $\epsilon=X+\xi_1+i\xi_2\in \Gamma(\gr(\pi))$ such that $\epsilon|_N=0$ and $X$ is Euler-like along $N$. Then $B+i\omega$ is an extension of the complex symplectic vector bundle $((\nu_N)_{\C}, \Omega_{\C})$. Furthermore, $\T_{\C}\psi: \T_{\C}\nu_N\to \T_{\C}M$ restricts to an isomorphism of complex Dirac structures
    \begin{equation*}
   L(B+i\omega)= (p^!\gr(\pi_N))^{B+i\omega}\to \gr(\pi)|_U,
\end{equation*}
    where $L(B+i\omega)$ is the local model associated to $B+i\omega$.
\end{theorem}
\begin{proof}
Our proof is inspired in \cite{bursztyn2019splitting}. Since the image of $d^N\alpha$ lies inside $\Ann TN$, the bundle $TN$ is inside $\ker (B+i\omega)$ and by $\C$-linearity we have $\TCN\subseteq\ker(B+i\omega)$.

By equation \eqref{pi_1_sub}, have the following decompositions
\begin{equation}\label{pi1_decomposition}
    TM|_N=TN\oplus \nu_N=T\nu_N|_N
\end{equation}
and its dual 
\begin{equation}\label{pi1_dual_decomposition}
    T^*M|_N=T^*N\oplus \nu^*_N=T^*\nu_N|_N,
\end{equation}
where in the last decomposition we are using the fact that $\Ann TN \cong \nu^*_N$.
Both decomposition give place to the following decompositions:
\begin{equation*}
    \T M|_N=\T N\oplus (\nu_N\oplus\nu^*_N)=\T\nu_N|_N
\end{equation*}
\begin{equation*}
        \T_{\C} M|_N=\T_{\C} N\oplus (\nu_N\oplus\nu^*_N)_{\C}=\T_{\C}\nu_N|_N
\end{equation*}
We recall that $d^N X=\pi^{\shrp}(d^N\epsilon)$, and so the image of $d^NX$ lies on $\pi^{\shrp}(\Ann TN)=\pi_1(\Ann TN)=\nu_N$. Since $T\psi|_N=d^NX$, we have that $T\psi|_N$ preserves the decomposition \eqref{pi1_decomposition} and so all the other decompositions.

Since $\gr(\pi)|_N\subseteq \T_{\C}M|_N=\T_{\C} N\oplus (\nu_N\oplus\nu^*_N)_{\C}$, we can project $\T_{\C} N$ and $(\nu_N\oplus\nu^*_N)_{\C}$. Note that, the projection of $\gr(\pi)|_N$ to $\T N$ is $\gr(\pi_N)$ and since $N$ is complex cosymplectic transversal the projection of $\gr(\pi)|_N$ onto $(\nu_N\oplus\nu^*_N)_{\C}$ is $\gr(\Omega_{\C})$, therefore 
\begin{equation*}
\gr(\pi)|_N=\gr(\pi_N)\oplus \gr(\Omega_{\C}).
\end{equation*}

Since $p^!\gr(\pi_N)\subseteq \T_{\C}\nu_N=\T_{\C} N\oplus (\nu_N\oplus\nu^*_N)_{\C}$, as in the proof of Proposition \ref{local_model}, we have that
\begin{equation*}
    p^!\gr(\pi_N)|_N=\gr(\pi_N)\oplus(\nu_N)_{\C}
\end{equation*}
    
   Due to the inclusion $\TCN\subseteq\ker(B+i\omega)$, we have 
   \begin{equation*}
    (B+i\omega)|_{T_{\C}\nu_N|_N}=(B+i\omega)|_{(\nu_N)_{\C}\oplus T_{\C}N}=(B+i\omega)|_{(\nu_N)_{\C}}
   \end{equation*}
Note that
\begin{align*}
   e^{B+i\omega}p^!\gr(\pi_N)|_N &=e^{B+i\omega}(\gr(\pi_N)\oplus(\nu_N)_{\C})\\
&=\gr(\pi_N)\oplus\gr((B+i\omega)|_{(\nu_N)_{\C}}).\\
\end{align*}
By Proposition \ref{normal_cx_dirac}, $\T\psi(e^{B+i\omega}p^!\gr(\pi_N)|_N)=\psi^!\gr(\pi)|_N$ and since 
\begin{equation*}
\psi^!\gr(\pi)|_N=\psi^!(\gr(\pi_N)\oplus\gr(\Omega_{\C})),
\end{equation*}
we obtain that $(B+i\omega)|_{(\nu_N)_{\C}}=\Omega_{\C}$. Therefore,  $B+i\omega$ is an extension of $((\nu_N)_{\C},\Omega_{\C})$. The second part of the proposition follows from Proposition \ref{normal_cx_dirac}.
\end{proof}

\begin{remark}
Note $\Omega_{\C}=(B+i\omega|_{(\nu_N)_{\C}})^{-1}$ is closed, since it comes from $\pi$ that is integrable and so $\Omega$ is closed as well. Nevertheless, $\Omega$ comes from $\pi_1$ that is not necessarily integrable, we can say that $\pi_1$ is Poisson on the normal directions of $N$.
\end{remark}

As a corollary, we have a version for the Weinstein splitting theorem for complex Poisson bivectors. 
\begin{corollary}\label{cx_weinstein_splitting}
Let $\pi=\pi_1+\pi_2$ be a complex Poisson bivector with constant order and let $m\in M$ be a point. If there exists a mixed submanifold $N\xhookrightarrow{\iota}M$ passing through $m$ and $\nu_N=N\times P$, then there exists a neighborhood $U$ of $m$ in $N\times P$ and a complex symplectic two-form $\Omega_P\in \Omega^2_{\C}(P)$ such that 
    \begin{equation*}
        \pi|_U=\pi_N+(\Omega_P)^{-1}.
    \end{equation*} 
\end{corollary}
\begin{proof}
For simplicity we assume that $P=\R^{2k}$, $M=N\times \R^{2k}$ and $m=(n,0)\in N\times \R^{2k}$. Choose a section $\epsilon=X+\xi_1+i\xi_2\in \Gamma(\gr(\pi))$ as in Theorem \ref{normalform_cx_poisson} and construct a closed complex two-form $B+i\omega$ which provide the local model $L(B+i\omega)\cong\gr(\pi)$. Since $\Ann T_{\C}N\subseteq\ker(B+i\omega)$ and $B+i\omega$ is closed, the two-form $B+i\omega$ is vertical with respect to the distribution $TN\times \{0\}$ in $N\times \R^{2k}$, thus it descends to a closed complex two-form $\Omega_P$ in $\R^{2k}$. Note that $\Omega_P|_0$ corresponds to the symplectic complex two-form on the fibers of $(\nu_N,\widetilde{\Omega})$, that is $\Omega_P|_0=\widetilde{\Omega}|_{(n,0)}$. Consequently in a neighborhood of $(n,0)$, the complex two-form $\Omega_P$ is invertible. Hence, by Theorem \ref{normalform_cx_poisson}, 
   \begin{equation*}
       \gr(\pi)=e^{B+i\omega}(\gr(\pi_N)\times T_{\C}P)=\gr(\pi_N)\times \gr(\Omega_P).
   \end{equation*}
   And thus the corollary holds.
\end{proof}

Note that $(\Omega_P)^{-1}$ is not written in a local form since we do not know the local form for complex symplectic structures.

\begin{lemma}[\cite{aguero2022complex}]\label{loc_model}
Let $L$ be a complex Dirac structure with constant real index $r$ and order $s$, a point $p\in M$ of type $k$ and let $N\xhookrightarrow{\iota}M$ be a submanifold such that $T_{p}N\oplus \Delta|_{p}=T_{p}M$. Then there exist a neighborhood $U$ of $p$, a diffeomorphism $\psi:N\times \R^{2(n-k)+r-s}\to U$ sending ${N\times 0}$ to ${N}$ and $\{ p\}\times\R^{2(n-k)+r-s}$ to the presymplectic leaf passing through $p$, and a section $\varepsilon=X+\xi_1+i\xi_2\in \Gamma(\psi^{!}L)$ such that $\varepsilon|_{N\times 0}=0$, $X$ is Euler-like, 
\begin{equation*}
    \xi_2=\sum^{n-k}_{i=1}q_{i}dp_{i}-p_{i}dq_{i}
\end{equation*}
and $\beta\in \Gamma(T^{*}M)$ where $(q_{1},\ldots,q_{n-k},p_{1},\ldots p_{n-k},z_{1},\ldots z_{r-s})$ are the coordinates of $\R^{2(n-k)+r-s}$.
\end{lemma}
\begin{corollary}
Let $\pi=\pi_1+\pi_2$ be a complex Poisson bivector with constant real index $s$ and constant order $s$ and let $p\in M$ be a point. If there exists a mixed submanifold $N\xhookrightarrow{\iota}M$ passing through $p$ and $\nu_N=N\times P$, then there exists a neighborhood $U$ of $p$ and $B\in \Gamma(\wedge^2T^*M|_U)$ such that 
    \begin{equation*}
        \pi^{-B}=\pi_N+(i\omega_{\C})^{-1},
    \end{equation*}
    where $\omega_{\C}$ is the complexification of the canonical symplectic form in $P\cong \R^{2(n-k)}$ and $\pi^{-B}$ denotes the gauge transformation of $\pi$ with respect to $-B$. 
\end{corollary}
\begin{proof}
  The corollary follows by applying the choice of local section from Lemma \ref{loc_model} to Theorem \ref{normalform_cx_poisson}.   
\end{proof}

\section{Conclusions}
Complex Poisson bivectors have a very rich geometry. Under suitable regularity conditions we can associate to them many Lie algebroids. For example:
\begin{enumerate}
\item $(T^*_{\C}M,[\cdot,\cdot]_{\pi}, \pi^{\shrp})$
    \item $(\ker\pi_2,[\cdot,\cdot]_{\pi_1}, \pi_1^{\shrp})$ inside $T^*_{\pi_1}M$ and $(\ker\pi_1,[\cdot,\cdot]_{\pi_2}, \pi_2^{\shrp})$ inside $T^*_{\pi_2}M$
    \item $(\ker\pi_1)_{\C}$ and $(\ker\pi_2)_{\C}$ as Lie subalgebroids of $(T^*_{\C}M)_{\pi}$
    \item $K_{\pi}=\gr(\pi)\cap \T M$
\end{enumerate}
plus all the Lie algebroids associated to the complex Lie algebroid $(T^*_{\C}M)_{\pi}$ (see \cite{aguero2024complexliealgebroidsconstant}). 
We can also associate several distributions like 
\begin{equation*}
    E_{\pi}, \:\:\deltpi, \:\:\Dpi,\:\:\ker\pi_1,\:\:\ker\pi_2, \:\:\ima\pi_1, \:\:\ima\pi_2,\:\:\pi_2(\ker\pi_1)\:\text{ and } \:\pi_1(\ker\pi_2)
\end{equation*}
some of them integrable under mild regularity conditions.

Finally, we can associate several lagrangian families:
\begin{equation*}
    L_{\pi}=\gr(\pi),\:\: L_{\pi_1,\pi_2}, \:\:\gr(\pi_1),\:\:\gr(\pi_2), \:\:\widehat{L_{\pi}}, \:\:\widecheck{L_{\pi}}\:\text{ and }\:\widetilde{L_{\pi}}
\end{equation*}
some of them Dirac structures under some regularity conditions.

More importantly, complex Poisson structures provide an unexplored bridge between bi-Hamiltonian systems and Dirac geometry.\footnote{There are other connections between bi-Hamiltonian system and Dirac geometry, e.g., Dirac pairs \cite{kosmann2012dirac} and concurrence of Dirac structures \cite{frejlich2024dirac}.}

\appendix
\section{Complex linear algebra and complex Cartan calculus}\label{apcxcartan}
\subsection{Complex linear spaces}
We recall some well-known facts of complex linear algebra. Let $V$ be a real vector space, we have the maps 
\begin{align*}
    \re:\{\text{complex subspaces of $V_{\C}$}\}&\to \{\text{real subspaces of $V$}\}\\
    S\subseteq V_{\C}&\mapsto \re S=S\cap V
\end{align*}
\begin{align*}
    \cdot_{\C}:\{\text{real subspaces of $V$}\}&\to \{\text{complex subspaces of $V_{\C}$}\}\\
    W&\mapsto W_{\C}
\end{align*}
These maps are not inverse one of the other. 

Let $U$ be a complex vector space, denote by $U_{\R}$ to its realification (restriction of scalar to $\R$). We will denote by $U^*$ its complex dual and by $(U_{\R})^*$ its real dual.

Let $\xi: U\to \C$, $\C$-linear, $\xi(u)=\eta_1(u)+i\eta_2(u)$, using that $\xi(iu)=i\xi(u)$ we get $\eta_1 \circ j=-\eta_2$, where $j:U\to U$, $j(u)=iu$.
Hence, we obtain the following map
\begin{equation*}
\eta\in (U_{\R})^*\to \eta-i\eta\circ j\in (U^*)_{\R},    
\end{equation*}
that is in fact a $\R$-isomorphism.
So we have:
\begin{proposition}
    $(U_{\R})^*\cong (U^*)_{\R}$
\end{proposition}
Consider the following map:
\begin{align*}
    \widehat{}\:\::(V^*)_{\C}\to (V_{\C})^*,\:\:\:\:\:\:\:
    \xi=\xi_1+i\xi_2&\mapsto \widehat{\xi},
\end{align*}
where $\widehat{\xi}$ is given by:
\begin{equation*}
    \widehat{\xi}(X+iY)=\xi_1(X)-\xi_2(Y)+i(\xi_2(X)+\xi_1(Y)).
\end{equation*}
Note that for pure real and pure imaginary elements we have the following: 
\begin{equation*}
    \widehat{\xi_1}=(\xi_1)_{\C}\:\:\text{and}\:\:\widehat{(i\xi_2)}=i(\xi_2)_{\C}.
\end{equation*}

\begin{proposition}\label{multicxisom}
    The map $\:\:\widehat{}\:\:$ is an isomorphism and it extends to an isomorphism from $(\wedge^k V)_{\C}$ to $\wedge^k V_{\C}$.
\end{proposition}

Let $V$ and $W$ be two real vector spaces, and $\Phi: V_{\C}\to W_{\C}$ a $\C$-linear map. Consider $S$ and $S'$, real vector subspaces of $V$ and $W$, respectively. If $\Phi(S)=S'$, then $\Phi(S_{\C})=S'_{\C}$. However, the converse is not true, if $\Phi(S_{\C})=S'_{\C}$, then we do not necessarily have that $\Phi(S)=S'$. 
\subsection{Complex bilinear operators}
Let $S$ be a complex vector space, $W$ a real vector space and $B\in Mult_k(S, W_{\C})$ a $k$-multilinear operator. We have the following decomposition
\begin{equation}\label{decom1}
B(w_1,\ldots,w_k)=\widetilde{B}_1(w_1,\ldots,w_k)+i\widetilde{B}_2(w_1,\ldots,w_k).\end{equation}
However, note that $\widetilde{B}_1,\widetilde{B}_2$ are not $\C$-linear, in fact $\widetilde{B}_1,\widetilde{B}_2\in Mult_k(S_{\R},W)$ and satisfy the following property:
\begin{equation*}
\widetilde{B}_1(w_1,\ldots,iw_j, \ldots ,w_k)=-\widetilde{B}_2(w_1,\ldots,w_j, \ldots ,w_k).
\end{equation*}
\begin{lemma}\label{bivdecom}
 Let $V$ and $W$ be real vector spaces and $B:V_{\C}\times V_{\C}\to W_{\C}$ any $\C$-bilinear map and the decomposition of \eqref{decom1} $B=\widetilde{B}_1+i\widetilde{B}_2$. Consider $B_1=\widetilde{B}_1|_{V\times V}$ and $B_2=\widetilde{B}_2|_{V\times V}$. Then
 \begin{align*}
     B(X_1+iX_2, Y_1+iY_2)&= B_1(X_1,Y_1)-B_1(X_2,Y_2)-B_2(X_2,Y_1)-B_2(X_1,Y_2)\\
 &+i(B_2(X_1,Y_1)-B_2(X_2,Y_2)+B_1(X_2,Y_1)+B_1(X_1,Y_2)).
 \end{align*}
 In particular, for $X,Y\in V$, $B(X,Y)=B_1(X,Y)+iB_2(X,Y)$.
\end{lemma}
As an immediate consequence we obtain a characterization of complexification:
\begin{corollary}\label{realoperator}
Let $B:\underbracket{V_{\C}\times\ldots\times V_{\C}}_{k-times}\to W_{\C}$ be a $\C$-multilinear map. Then $B$ is a complexification if and only if $B_2=0$. 
\end{corollary}

\subsection{Complex Cartan calculus}
The complex Schouten bracket is defined in the usual way:
\begin{equation*}
    [\cdot,\cdot]_{Sch}:\Gamma(\wedge^p \TCM)\times\Gamma(\wedge^q \TCM)\to \Gamma(\wedge^{p+q-1}\TCM)
\end{equation*}
\begin{equation*}
    [X_1\wedge\ldots\wedge X_p, Y_1\wedge\ldots\wedge Y_q]=\sum_{i,j}(-1)^{i+j}[X_i,Y_j]\wedge X_1\wedge\ldots \widehat{X_i}\wedge\ldots\wedge X_p\wedge Y_1\wedge \ldots \widehat{Y_j}\ldots \wedge Y_q
\end{equation*}
and then we extend it by $\C$-linearity.
Note that, by Corollary \ref{realoperator}, it is the complexification of the real Schouten bracket and satisfies the same properties that the real Schouten bracket. 

We have two complex of complex differential forms:

The complex $(\Gamma(\wedge^{\bullet} T^*_{\C}M), \widetilde{d})$: for a complex differential form $\alpha\in \Gamma(\wedge^{k} T^*_{\C}M)$, we have the operator 
\begin{align*}
\widetilde{d}\alpha(Z_0,\ldots, Z_k)&=\sum_j (-1)^{j+1}Z_j(\alpha(Z_0, \ldots, \widehat{Z_j}, \ldots Z_k))\\
&+\sum_{i,j}(-1)^{i+j+1}\alpha([Z_i,Z_j], Z_0, \ldots, \widehat{Z_i},\ldots, \widehat{Z_j},\ldots, Z_k),
\end{align*}
where $Z_0,\ldots, Z_k\in  \Gamma(\TCM)$.

The complex $(\Gamma(\wedge^{\bullet} T^*M)_{\C}, d_{\C})$:
for a complex differential form $\alpha=\alpha_1+i\alpha_2\in \Gamma(\wedge^k T^*M)_{\C}$, we define the differential $d_{\C}$ as 
\begin{equation*}
d_{\C}\alpha=d\alpha_1+i d\alpha_2.
\end{equation*}

By Lemma \ref{multicxisom}, we have that the map 
\begin{equation*}
\hat{}:\Gamma(\wedge^k T^*_{\C}M)\to \Gamma(\wedge^k T^*M)_{\C}
\end{equation*}is an isomorphism. 
\begin{proposition}
The operators $\widetilde{d}$ and $d_{\C}$ are equivalent in the following sense:
\begin{equation*}
\begin{tikzcd}
\Gamma(\wedge^k T^*_{\C}M)\arrow[d, "\widetilde{d}"] \arrow[r, "\widehat{}"]
& \Gamma(\wedge^k T^*M)_{\C}\arrow[d,  "d_{\C}"] \\
\Gamma(\wedge^{k+1} T^*_{\C}M)\arrow[r, "\widehat{}"]
& \Gamma(\wedge^{k+1} T^*M)_{\C}.
\end{tikzcd}
\end{equation*}
    
\end{proposition}

Also consider the Lie derivative as 
\begin{equation*}
    L_{X+iY}(\beta_1+i\beta_2)=L_{X}\beta_1 -L_{Y}\beta_2 + i(L_{X}\beta_2 + L_{Y}\beta_1).
\end{equation*}

An straightforward computation proves the following:
\begin{proposition}[Complex Cartan magic formula]
Let $X+iY\in \Gamma(\TCM)$. Then, 
    \begin{equation*}
        L_{X+iY}=\iota_{X+iY}d_{\C}+d_{\C}\iota_{X+iY}
    \end{equation*}
\end{proposition}
\section{Dirac geometry}\label{apcxdirac}
Given a manifold $M$ we consider its  {\em generalized tangent bundle} $\T M:=TM\oplus T^*M$ with the canonical nondegenerate symmetric pairing
\begin{equation*}
\langle X+\xi, Y+\eta\rangle=\frac{1}{2}(\eta(X)+\xi(Y)),
\end{equation*}
 and the Courant-Dorfman bracket \cite{dorfman1987dirac} on $\Gamma(\T M)$
\begin{equation*}
[ X+\xi, Y+\eta]=[X,Y]+\lie_{X}\eta-\imath_{Y}d\xi,\:\:{\rm for}\:\:X+\xi,\:\: Y+\eta \in \Gamma (\T M).
\end{equation*}

Using extension of scalars, we extend the canonical pairing $\la\cdot,\cdot\ra$ and the Courant-Dorfmann bracket $[\cdot,\cdot]$ to $\TTCM$ and $\Gamma(\TTCM)$, respectively. Along this article we focus on structures defined on $\TTCM$. 
\begin{definition}
A {\em lagrangian family} is a correspondence $p\in M\mapsto L_p\subseteq \TTCM$, where $L_p$ is a lagrangian subspace of $\TCM|_p$ (we can define also lagrangian families on $\T M$ in the obvious way, and we still refer to them as lagrangian families). We say that the lagrangian is {\em smooth} if it defines a subbundle of $\TCM$ or $\T M$, in that case the lagrangian family is called {\em lagrangian subbundle}.
A {\em (complex) Dirac structure} is a lagrangian subbundle that is involutive with respect to the Courant-Dorfman bracket (that is, $[\Gamma(L),\Gamma(L)]\subseteq \Gamma(L)$).
\end{definition}

The {\em range distribution} $E$ of a lagrangian subbundle $L$ is 
\begin{equation*}
 E:=\pr_{TM}L\:\text{ or }\:E:=\pr_{\TCM}L,  
\end{equation*}
 which is smooth but not necessarily regular. There exists a  skew-symmetric bilinear map $\varepsilon_L:E\times E\to \K$ such that 
\begin{equation*}
    L=L(E,\varepsilon_L):=\{X+\xi\:|\: X\in E,\:\:\xi_{|E}=\imath_{X}\varepsilon_L\},
\end{equation*}
 where $\K=\R$ or $\C$. While the range distribution of a Dirac structure is always integrable to a possibly singular foliation, the range distribution of a complex Dirac structure is involutive but not integrable (until this moment we have not found a sense of the integration of an involutive regular complex distribution). Furthermore, in the case of Dirac structures $\varepsilon_L$ restricts to a presymplectic form on each leaf.
\begin{examples}
\begin{enumerate}\normalfont
    \item Let $\pi$ be a real or complex Poisson bivector. Then, the graph of $\pi$ is a Dirac or complex Dirac structure, respectively.
    \item Let $\F$ be a regular foliation where each leaf $S$ is equipped with a presymplectic two-form $\omega_S$, and this family of presymplectic forms is calibrated, i.e. there exists a closed two-form $\omega\in \Omega^2(M)$ such that $\iota^*_S\omega=\omega_S$. Then, $R=L(T\F,\omega)$ is a Dirac structure and $L=L((T\F)_{\C},i\omega_{\C})$ is a complex Dirac structure.

\end{enumerate}
\end{examples} 
$B$-fields and $\beta$-transformations
Two-forms act on the space of lagrangian families, in the following way: given $B\in \Omega^{2}(M)$ or in $\Omega_{\C}^{2}(M)$ and $L$ a real or complex lagrangian family, repectively
\begin{equation*}
    e^BL=\{X+\xi+\imath_{X}B\st X+\xi\in L\}
\end{equation*}
is the {\em $B$-transformation} of $L$ by $B$. When $B$ is closed and $L$ is real or complex Dirac, $e^{B}L$ is real or complex Dirac, respectively.

A {\em $\beta$-transformation} is an action of a bivector $\beta\in \Gamma(\wedge^2 TM)$ or in $\Gamma(\wedge^2 \TCM)$ over a real or complex lagrangian family, respectively, given by:
\begin{equation*}
    e^{\beta}L=\{X+\imath_{\xi}\beta+\xi\st X+\xi\in L\}.
\end{equation*}
\begin{definition}\label{ri_order}
The {\em real index} of a \underline{complex} lagrangian family $L$ is the $\mathbb{N}$-valued function
\begin{equation*}
    p\in M\mapsto \text{real-index}|_p L=L|_p \cap \T_p M. 
\end{equation*}
Additional to its range, a complex Dirac structure $L$ has associated the real distributions:
\begin{equation*}
    \Delta_L=\pr_{\TCM} L \cap TM \subseteq TM\text{ and } D_L=\pr_{TM} L \subseteq TM. 
\end{equation*}
We say that a complex Dirac structure has {\em constant order} whenever $D_L$ is a vector bundle.
\end{definition}
Moreover, any complex Dirac structure $L$ has associated the real distribution 
\begin{equation*}
L^{\rea}=L\cap(TM\oplus\CCM).
\end{equation*} 
\begin{proposition}[\cite{aguero2024complexliealgebroidsconstant}]\label{lreisalgebroid}
If $ \pr_{TM}L$ is a subbundle of $TM$, then $L^{\rea}$ is a Lie algebroid.
\end{proposition}

In the same way as with Dirac structures, complex Dirac structures have a tangent product: Let $L_1$ and $L_2$ be two lagrangian families, consider the space:
\begin{equation*}
    L_1 \star L_2=\{X+\eta_1+\eta_2\st X+\eta_1\in L_1\:\:\text{and}\:\: X+\eta_2\in L_2 \}.
\end{equation*}
Note that $L_1*L_2$ is not necessarily smooth but it is a lagrangian family. However, if $L_1$ and $L_2$ are complex Dirac structures and the smoothness condition is satisfied, then $L_1 * L_2$ is a complex Dirac structure.

There is an action of $\C$ on the space of lagrangian families, given by the following formula:
\begin{equation*}
    z\cdot L=\{X+z\xi\st X+\xi\in L\},
\end{equation*}
where $z\in \C$ and $L$ is a lagrangian family. Observe that this action satisfies the following properties:
\begin{align}\label{prop_cx_action}
    \overline{z\cdot L}&=\overline{z}\cdot\cL\\
    z\cdot(L_1\star L_2)&=(z\cdot L_1)\star(z\cdot L_2),
\end{align}
for $z\in \C$ and $L, L_1,L_2$ lagrangian families. Using the versatility of complex Dirac structures we can define other actions on the space of lagrangian families.

Any complex Dirac structure $L$ has associated a real lagrangian family that we denote by $\widehat{L}$. On way to obtain it is by the identity, \cite{frejlich2024dirac}
\begin{equation*}
    L\star(-1)\cdot\overline{L}=2i\cdot \widehat{L}_{\C}.
\end{equation*}

It is known the following identification:
\begin{equation}\label{lhatid}
\widehat{L}=\{ X+\xi\st \exists\eta\in T^*M \:\:\text{such that}\:\: X+i\xi+\eta\in L\}.
\end{equation}

Let $K=L\cap \T M$ be the space of real elements of $L$ and denote by $K^{\perp}$ to the orthogonal of $K$ with respect to the canonical pairing $\la\cdot,\cdot\ra$. Then, we have the following:

\begin{lemma}\label{proj_K}\cite[Lemma 4.8]{aguero2022complex}
For the distributions $K$ and $K^\perp$ we have
\begin{align*}
    \pr_{TM} K^\perp &= D\:\text{ and}\\ \pr_{TM} K &= \ker\varepsilon_{\widehat{L}}.
\end{align*}
\end{lemma}
Let $\varphi:M\to N$ be a smooth map and $L_M$ and $L_N$ complex Dirac structures on $M$ and $N$, respectively. The {\em backward image of $L_N$ under $\varphi$} is the lagrangian family:
\begin{equation*}
    (\varphi^! L_N)|_m=\{X+\varphi^* \xi\st T\varphi (X)+\xi\in L_N|_{\varphi(m)}\}\subseteq \TCM
\end{equation*}
and the {\em forward image of $L_M$ under $\varphi$}
\begin{equation*}
    (\varphi_! L_M)|_m=\{T\varphi (X)+\xi\st X+\varphi^*\xi\in L_M|_m\}\subseteq\varphi^*\TCN.
\end{equation*}
Note that the forward image is not a subbundle of $\TCN$. We say that $L_M$ is {\em $\varphi$-invariant} if for $m, m'\in M$ in the same $\varphi$-fiber, we have that $(\varphi_! L_M)|_m=(\varphi_! L_M)|_{m'}$. If $L_M$ is $\varphi$-invariant, the forward image $\varphi_! L_M$ defines a lagrangian family on $N$, that we shall call forward image and denote by $\varphi_! L_M$, if it is clear from the context.

As with real Dirac structures, under certain clean intersection conditions, we can prove the smoothness of the backward and forward image, see for example \cite{bursztyn2013brief}.

In the same way as in Example 5.10 of \cite{bursztyn2013brief}, we have the following result: 
\begin{lemma}\label{quotientdirac}
    If $\ker Tp\subseteq L\cap TM=\re\ker L$, then $L$ is $p$-invariant.
\end{lemma}

\section{Euler-like vector fields and Normal forms}\label{apnormal}
 The {\em Euler vector field} associated to a vector bundle $E\to M$, usually denoted by $\mathcal{E}\in \mathfrak{X}(E)$, is the vector field generated by the one-parameter group $s\mapsto \kappa_{e^{-s}}$, where the maps $\kappa_{t}:E\to E$ are given by $\kappa_{t}e=t\cdot e$, $t\in \R$ (in case $t=0$, $\kappa_0$ is the projection to the zero section).

Given a submanifold $N\xhookrightarrow{}M$, {\em the normal bundle of $N$} is the bundle 
\begin{equation*}
\nu_N=TM|_{N}/TN,
\end{equation*} with projection map by $p:\nu_N\to N$.

Any map of pairs $\varphi:(M',N')\to (M,N)$  has an associated map $\nu(\varphi):\nu_{N'}\to \nu_N.$
The following lemma is a known fact. 
Consider a vector field $X\in \mathfrak{X}(M)$ tangent to $N$, that is $X:(M,N)\to (TM,TN)$. Using the fact that $T\nu_N=\nu_{TN}$ (see \cite[Appendix A]{bursztyn2019splitting}) we have that $\nu(X)$ is a vector field on~$\nu_N$.

Let $p:E\to M$ be a vector bundle and $s\in \Gamma(E)$ such that $s|_N=0$. Then, we have a map of pairs $s:(M,N)\to (E,M)$. The {\em normal derivative of $s$} is the map
\begin{equation*}
    d^N s:\nu_N\to E|_N,
\end{equation*}
given by $d^N s=\nu(s)$, here we are using the identification $\nu_M=TE|_M/TM=E|_M$. Note that any bundle map from $\nu_N$ to $E|_N$ is the normal derivative of a certain section $s\in \Gamma(E)$ satisfying that $s|_N=0$.
\begin{definition}\label{defTN-EL}
Let $N$ be an embedded submanifold of $M$. A {\em tubular neighborhood embedding} is an embedding 
$\psi:\nu_N\to M$ that takes $N$ as the zero section of $\nu_N$ to $N\subseteq M$ and $\nu(\psi)=Id$, where $\nu(\psi)$ is induced by $\psi:(\nu_N,N)\to (M,N)$ (here we are using the identification $\nu(\nu_N,N)=\nu_N$).
A vector field $X\in\mathfrak{X}(M)$ is called {\em Euler-like} (along $N$) if it is complete, $X|_{N}=0$ and $\nu(X)$ is the Euler vector field of $\nu(M,N)$.
\end{definition}
Given a submanifold $N$ of $M$, there is a one-to-one correspondence between Euler-like vector and tubular neighborhood embedding, see for example \cite{bursztyn2019splitting}. 
\begin{lemma}\cite[Lemma 3.9]{bursztyn2019splitting}\label{existence_esp_section}
Let $(A, \rho)$ be RAVB over $M$, and $N\subseteq M$ a submanifolds transversal to $\rho$.
Then there exists a section $\epsilon\in \Gamma(A)$ with $\epsilon|_N = 0$, such that $X = \rho(\epsilon)$ is Euler-like along~$N$.
\end{lemma}

The way to construct $\epsilon$ is the following. Consider, the exact sequence
\begin{equation*}
\xymatrix{
0\ar[r]&\iota^! A\ar[r]& A|_N\ar[r]&\nu_N\ar[r]&0},
\end{equation*}
where the map $A_|N\to \nu_N$ is given by the anchor map composed with the quotient map. Any section $\nu_N\to A|_N$ of the exact sequence comes from the normal derivative of a section $\epsilon\in \Gamma(E)$ such that $\epsilon|_N=0$. It is proved that $X=\rho(\epsilon)$ is Euler-like.
\bibliographystyle{plain}
\bibliography{references}

\end{document}